\documentclass
{amsart}

\usepackage{amsmath, hyperref}
\usepackage{amssymb}
\usepackage{amsfonts}

\usepackage[latin1]{inputenc}
\usepackage{latexsym, 
}

\usepackage{graphicx}
\usepackage{tikz,quiver}
\usepackage{adjustbox}

\usepackage{amsmath}
\usepackage{xcolor}
\newcommand{\Tfrac}[2]{%
  \ooalign{%
    $\genfrac{}{}{1.2pt}1{#1}{#2}$\cr%
    $\color{white}\genfrac{}{}{.4pt}1{\phantom{#1}}{\phantom{#2}}$}%
}
\newcommand{\Dfrac}[2]{%
  \ooalign{%
    $\genfrac{}{}{1.2pt}0{#1}{#2}$\cr%
    $\color{white}\genfrac{}{}{.4pt}0{\phantom{#1}}{\phantom{#2}}$}%
}

\usepackage{enumerate}  
\usepackage{mathtools}  
\usepackage{bm}         
\usepackage{url}        

\usepackage{tikz}

\usepackage{bussproofs}

\usetikzlibrary[arrows,snakes]
\tikzstyle{bull}=[circle,draw=black,fill=black!80]
\tikzstyle{holl}=[circle,draw=black]

\usepackage{microtype}
\makeatletter
\def\MT@register@subst@font{\MT@exp@one@n\MT@in@clist\font@name\MT@font@list
   \ifMT@inlist@\else\xdef\MT@font@list{\MT@font@list\font@name,}\fi}
\makeatother

\makeatletter
\def\@listi{\leftmargin\leftmargini
              \parsep 0\p@ \@plus2\p@
              \topsep 2\p@ \@plus1\p@
              \itemsep 2\p@ \@plus1\p@}
\let\@listI\@listi
\@listi
\makeatother

\newtheorem{theorem}{Theorem}[section]
\newtheorem{lemma}[theorem]{Lemma}
\newtheorem{proposition}[theorem]{Proposition}

\newtheorem{corollary}[theorem]{Corollary}

\theoremstyle{definition}
\newtheorem{definition}[theorem]{Definition}

\newtheorem{example}[theorem]{Example}

\newcommand{\la}{\langle}
\newcommand{\ra}{\rangle}

\newcommand{\Al}[1][A]{\ensuremath{\mathbf{#1}} }

 \newcommand{\K}{\mathbb{K}}

\renewcommand{\phi}{\varphi}

\newcommand{\A}{\Al}

 \newcommand{\Fm}{\Al[Fm]}

 \newcommand{\DNE}{\mathsf{DN}}
 \newcommand{\Eq}{\mathsf{E}} 
 \newcommand{\SDM}{\mathbb{SDM}}
 \newcommand{\sdm}{\mathsf{SDM}}
 \newcommand{\OC}{\mathbb{O}}
 \newcommand{\DN}{\mathbb{DN}}
 \newcommand{\BA}{\mathbb{B}}
 \newcommand{\DM}{\mathbb{DM}} 
 \newcommand{\PDL}{\mathbb{PL}}
  \newcommand{\VV}{\mathbb{V}}

 \newcommand{\PP}{\mathcal{P}}

 \newcommand{\Alg}{\mathsf{Alg}}

 \let\Omega=\varOmega

 \let\Gamma=\varGamma

 \let\Lambda=\varLambda

 \let\Sigma=\varSigma

 \let\Psi=\varPsi

 \let\Delta=\varDelta

 \let\Pi=\varPi

 \let\Theta=\varTheta

\newcommand{\Ru}{\mathcal{R}}
\newcommand{\Su}{\mathcal{S}}
\newcommand{\ru}{\mathsf{r}}

\newcommand{\Leibniz}{{\boldsymbol{\varOmega}}}

\newcommand{\Mt}{\mathbb{M}}

\newcommand{\nnot}{\mathop{\sim}}                  
\renewcommand{\neg}{\nnot}

\newcommand{\alg}[1]{\mathbf{#1}}                  

 \let\vDash=\models

\let\epsilon=\Phi

\newcommand{\eval}[2][\right]{\relax
   \ifx#1\right\relax \left.\fi#2#1\rvert}

\newcommand{\tuple}[1] {{\langle #1\rangle}}

\begin{document}

\title{ 
Finite axiomatizability of logics \linebreak
of distributive lattices with negation
  }

\author{S\'ergio Marcelino}
\address{
SQIG - Instituto de Telecomunica\c c\~oes,
Departamento de Matem\'atica, Instituto Superior T\'ecnico, 
Universidade de Lisboa,  Lisboa, Portugal
}
\email{sergiortm@gmail.com}
%
\author{Umberto Rivieccio}
\address{Departamento de Inform\'atica e Matem\'atica Aplicada,
Universidade Federal do Rio Grande do Norte,
Natal (RN), Brasil}
\email{urivieccio@dimap.ufrn.br}
\thanks{Research funded by FCT/MCTES through national funds and when applicable co-funded by EU under the project UIDB/EEA/50008/2020
and by the Conselho Nacional de Desenvolvimento Científico e Tecnol\'ogico (CNPq, Brazil), under the grant 313643/2017-2 (Bolsas de Produtividade em Pesquisa - PQ)}


\maketitle

\begin{abstract}
This paper focuses on  order-preserving logics defined from varieties
of distributive lattices with negation, and in particular on the problem of
whether these can be axiomatized 
by means of finite Hilbert calculi. 
On the side of negative results, we  provide a syntactic condition on the equational presentation
of a variety that entails 
 failure of finite axiomatizability for the corresponding logic.
An application of this result is that
the logic of all distributive lattices with negation is not finitely axiomatizable;
likewise, 
we establish that
the order-preserving  logic 
of the variety of all Ockham algebras is also
 not finitely axiomatizable.
 On the positive side, we show that
 an arbitrary subvariety of semi-De Morgan algebras 
is axiomatized by a finite number of equations
 if and only if
 the corresponding order-preserving logic is
  axiomatized by a finite Hilbert calculus.
This equivalence also holds for every subvariety of a
 Berman variety of Ockham algebras.
We obtain, as a corollary, a new proof that the implication-free
  fragment of intuitionistic logic is finitely axiomatizable, as well as a new 
  Hilbert calculus for it. 
  Our proofs are constructive in that they 
  allow us to effectively convert an equational presentation
  of a variety of algebras into
a Hilbert calculus for the corresponding order-preserving logic, and viceversa. 
 We also consider the assertional 
  logics associated to the above-mentioned varieties,
 showing in particular that 
the assertional logics of finitely axiomatizable subvarieties of semi-De Morgan algebras are finitely axiomatizable as well.
\end{abstract}


 
 \section{Introduction}
 \label{sec:intro}
 
 In the present paper, we study logics associated to subvarieties of 
 the class $\DN$ of \emph{distributive lattices with negation}  (Definition~\ref{def:dn}) considered
 for instance in the papers~\cite{Ce99,Ce07}. 
$\DN$ is a variety that  includes many well-known classes of algebras of non-classical logics,
such as (semi-)De Morgan algebras, Stone algebras, pseudo-complemented distributive lattices
and Ockham algebras,
 therefore providing  a common semantical framework for the study of the corresponding logics.

 We will be mostly concerned with the order-preserving logics associated 
 to the above-mentioned varieties, focusing in particular  on the issue of 
 whether
 they can be axiomatized or not by means of a Hilbert calculus
 consisting of finitely many rule schemata; if this is the case,
 the logic will be called \emph{finitely based}.
 
On the side of negative results, 
we are going to show that the order-preserving logic associated to the variety 
$\DN$ 
is not finitely based;
the same holds for the order-preserving logic of all Ockham algebras (Definition~\ref{def:sdmock}).
 Indeed, we will give a syntactic criterion regarding the equations that axiomatize
 (relatively to  $\DN$)
a variety $\VV \subseteq \DN$    
 implying that the same holds for the corresponding logic.  
On the positive side, we will 
show how to obtain a finite Hilbert calculus that is complete with respect to 
 the logic
of semi-De Morgan algebras, entailing  that the latter is finitely based.
The same techniques  will allow us to obtain finite calculi for the logics associated to
 so-called Berman varieties of Ockham algebras~\cite{Ber77}.
As a corollary of our results, we will also obtain a finite axiomatization 
for the logic of pseudo-complemented distributive lattices (i.e.~the implication-free fragment
of intuitionistic logic) alternative to the one introduced in~\cite{ReV94}.

Our proof strategies 
are discussed in more detail in Sections~\ref{sec:dn} and \ref{sec:sdm}, but we 
give here an introductory account on the  
finite axiomatizability problem for order-preserving logics and the difficulties one faces.
First of all, let us clarify the meaning of the terms ``order-preserving logic'' and ``finite Hilbert calculus''. 

Let $\K$ be a class (say, a variety)
of algebras such that each algebra $\A \in \K$ has a bounded lattice
reduct $\la A; \land, \lor, \bot, \top \ra$
One of the standard ways of associating 
a (finitary) Tarskian logic $ \vdash^{\leq}_{\K} $  to $\K$  is the following.
One lets $\emptyset \vdash^{\leq}_{\K}  \phi$ if and only if
the equation $\phi \approx \top$ is valid in $\K$ and,
for all $\Gamma \cup \{ \phi \} \subseteq Fm$ such that $\Gamma \neq \emptyset$, 
one lets
$
\Gamma \vdash^{\leq}_{\K} \phi
$
iff
there is a natural number $n$ and formulas
$\gamma_1, \ldots, \gamma_n \in \Gamma$
such that 
the equation 
$  \gamma_1 \land \ldots \land \gamma_n \land \phi \approx \phi $
is valid in $\K$. 
Thus $\vdash^{\leq}_{\K} $ is by definition a finitary logic,
called
the \emph{order-preserving logic} of the class $\K$.
Note that
$\vdash^{\leq}_{\K} $ coincides with the logic defined by the class of matrices
$
\{ \la \A, F \ra  :  \A \in \K, \, F \subseteq A \text{ is a non-empty lattice filter of } \A 
\}$.
Other logics may of course be defined from $\K$, for instance, the class of matrices
$\{ \la \A, \{ \top \} \ra  :  \A \in \K
\}$ also determines a (stronger) logic associated to 
$\K$.
Following~\cite{Ja06a}, we call   this 
 the \emph{$\top$-assertional logic} of the class of algebras $\K$
(denoted $\vdash^{\top}_{\K} $)
and will be considered in Section~\ref{sec:top}.

By a \emph{Hilbert calculus} we mean a logical calculus whose every rule schema
is a pair 
$\frac{\Gamma}{\phi}$
where $\Gamma$ is a finite (possibly empty) set of formulas and $\phi$ is a formula;
we say that such a calculus is \emph{finite} when it consists of finitely many rule schemata. 
Following~\cite[p.~607]{R91}, we call a
logic that can be axiomatized by a finite Hilbert  \emph{finitely based}. 
In~\cite[Sec.~2.1]{SDMWollic} the authors introduce a finite calculus for the order-preserving logic
of the variety $\SDM$ of semi-De Morgan algebras (Definition~\ref{def:sdmock}). This, however,
is not a Hilbert calculus \emph{stricto sensu}, because it involves sequent-style rule schemata
such as the following: from $\la \phi, \psi \ra$ infer $\la \nnot \psi, \nnot \psi \ra$. The ``axioms'' of the calculus introduced in~\cite{SDMWollic}, on the other hand, are examples of what are usually called (single-premiss)  
 Hilbert rules. 
Finite Hilbert calculi for the order-preserving logics
of De Morgan algebras ($\DM$) and pseudo-complemented distributive lattices $(\PDL)$
can be found
in the papers~\cite{F97,ReV94}. We note in this respect 
that $\vdash^{\top}_{\PDL} = \, \vdash^{\leq}_{\PDL} $, while 
$\vdash^{\top}_{\DM} $
is strictly stronger  than $\vdash^{\leq}_{\DM} $,  
which is the well-known Belnap-Dunn logic.
$\vdash^{\top}_{\DM} $ is the \emph{Exactly True Logic} introduced
and axiomatized by means of a Hilbert calculus in~\cite{PiRi}; see also~\cite{Ri,AlPrRi}
.

 A closer look at the order-preserving logic $\vdash^{\leq}_{\SDM} $
 associated to
 semi-De Morgan  algebras 
 explains the choice of a hybrid calculus in~\cite{SDMWollic},
as well as the challenge one faces when trying to axiomatize $\vdash^{\leq}_{\K} $ (for 
$\K \subseteq \DN$) by means of a calculus that is Hilbert in  the  strict sense. 
In fact, the consequence relation of 
each order-preserving logic $\vdash^{\leq}_{\K} $
corresponds to the lattice order
on $\K$,
in the sense that one has
$\phi \vdash^{\leq}_{\K} \psi$ if and only if the inequality $\phi \leq \psi$ (taking the latter as a shorthand for the 
equation $\phi \land \psi \approx \phi$) is valid in $\K$. 
Such a partial order relation on each 
$\A \in \K$
enjoys certain (meta-)properties that need to be mirrored by the logical calculus.
Indeed, every order-preserving logic $\vdash^{\leq}_{\K}$
is selfextensional (see Section~\ref{sec:dn}); moreover, observe that, 
if $\K \vDash \phi \leq \psi$, then $\K \vDash \nnot \psi \leq \nnot \phi$, but also
$\K \vDash \phi \lor \gamma \leq \psi \lor \gamma$ for every $\gamma \in Fm$, and so on. 

In~\cite{SDMWollic}, the above meta-properties are imposed 
by adding suitable sequent-style rule schemata
such as the one mentioned above (from $\la \phi, \psi \ra$ infer $\la \nnot \psi, \nnot \phi \ra$).
As is well known, pure Hilbert calculi (\emph{stricto sensu}) 
lack the expressive power needed to directly impose 
such (meta-)properties,
which is one of the  reasons of   interest in more expressive (e.g.~sequent-style) calculi.   
However, Hilbert calculi also allow for more fine-grained analyses of logics and,
being very close to the algebraic semantics, they are  generally better suited for the study of logics from an algebraic point of view
 (see e.g.~\cite[p.~414]{F97}).

A first approach to the above-mentioned axiomatizability problem  suggests the following strategy. Take a basic set of Hilbert rule schemata
$\Ru$ (which are sound w.r.t.~$\vdash^{\leq}_{\K} $) and recursively close it  under the application of rule schemata
as follows: whenever $\la \phi, \psi \ra \in \Ru$, add to $\Ru$ also 
$\la \nnot \psi, \nnot \phi \ra $,
$\la \phi \lor \gamma, \psi \lor \gamma \ra $, etc. 
Such a process is indeed bound to succeed, and allows one to
 show that the derivability relation $\vdash_\Ru$  thus obtained coincides with
$\vdash^{\leq}_{\K} $. The non-trivial question is  whether some \emph{finite}
subset 
$\Ru_0 \subseteq \Ru$ also suffices or not. 
The main result of the present paper consists in providing a  sufficient condition for the negative result to hold
  as well as
a few   conditions that are sufficient for ensuring a positive answer.
As we shall see, the answer relies crucially on the soundness of certain
rule schemata.

%
%
%
%
%
%
%
%
%
%

We note for the algebraic logician that the logics considered in the present paper are not algebraizable in the sense of Blok and Pigozzi,
and indeed they are easily shown to be non protoalgebraic either (see e.g.~\cite{FJa09} for the relevant definitions).
This is one of the challenges of our study, for 
one cannot rely on the existence of the translations between
equations and formulas  that are provided by the general theory of algebraizable logics.
Thus, 
in this setting,
there is no standard recipe 
for obtaining a Hilbert axiomatization
of a given logic from
an equational presentation of 
the corresponding 
class of algebras.
Also, no isomorphism is readily available between (say)
the lattice of subquasivarieties of $\DN$
and the lattice of finitary extensions of $\vdash^{\leq}_{\DN} $ (but see Theorem~\ref{thm:ramon} in Section~\ref{sec:var}).

The paper is organized as follows. Section~\ref{sec:var} collects
the fundamental definitions on algebras and logics, as well as a few useful lemmas.
In Section~\ref{sec:dn} we give a recipe for obtaining a (potentially infinite)
Hilbert axiomatization for the logic $\vdash^{\leq}_{\K} $ for each class $\K \subseteq \DN$.
We investigate conditions entailing that the above-mentioned axiomatization must be infinite, and in particular
we show that $\vdash^{\leq}_{\DN} $ is not finitely based; the same holds for
the logic $\vdash^{\leq}_{\OC} $
of the variety of all Ockham algebras (Definition~\ref{def:sdmock}).
By contrast, we show in Section~\ref{sec:sdm} that, for an arbitrary variety $\K \subseteq \SDM$,
where $\SDM$ is the class of semi-De Morgan algebras (Definition~\ref{def:sdmock}),
the logic $\vdash^{\leq}_{\K} $ is finitely based if and only if $\K$ is axiomatized by a finite
number of equations (in particular, $\vdash^{\leq}_{\SDM} $ is itself finitely based).
In Section~\ref{sec:ber} we adapt our proof techniques to show
that, unlike the whole variety $\OC$ of Ockham algebras, every Berman subvariety $\OC_n^m \subseteq \OC$
determines a logic $\vdash^{\leq}_{\OC_n^m} $ that is finitely based. 
In Section~\ref{sec:top}
we briefly consider $\top$-assertional logics associated to varieties of distributive lattices with negation,
showing in particular that $\vdash^{\top}_{\SDM} $ is finitely based. 
Lastly, Section~\ref{sec:conc} contains some
concluding remarks and suggestions for further research. 

\label{th:subva}

 \section{Algebraic and logical preliminaries} 
\label{sec:var}

\subsection{Algebras}

We  adopt the standard conventions and notation of modern universal algebra,
for which we refer the reader to~\cite{BuSa00}. All algebras considered in the present paper
are bounded (distributive) lattices (Definition~\ref{def:lat}) enriched with a unary 
\emph{negation} operation $\neg$ on which different requirements will be imposed,
giving rise to the various classes of interest.  The algebraic (as well as the logical) language $\{ \land, \lor, \neg, \bot,\top\}$, consisting of a conjunction (interpreted as the lattice meet on algebras), a disjunction (the join), a negation and truth constants (the top and bottom of the lattice) 
will stay fixed throughout the paper.
We shall denote by $\Fm$ the algebra of formulas over this language, freely generated 
by a denumerable set of variables (denoted $x,y,z$ etc.), and by $Fm$ the  universe of this algebra. 
We shall mostly be interested in equational classes of algebras, i.e.~\emph{varieties}. 
An \emph{equation} is a pair of algebraic terms $\la \phi, \psi \ra \in Fm \times Fm$,
and
every set $\Eq \subseteq \PP (Fm \times Fm)$ of equations determines a variety which will be denoted by
$\VV_{\Eq}$.

 \begin{definition}[\cite{BuSa00}]
\label{def:lat}
A \emph{bounded
distributive lattice} is an algebra 
$\Al = \la A; \land, \lor,\bot,\top \ra$ of type $\la 2,2,0,0 \ra$
such that 
the following equations are satisfied: 
\begin{enumerate}[(L1)]
\item \label{Itm:L1}  $x \lor y \approx y \lor x$ \qquad $x \land y \approx y \land x$. 
\item \label{Itm:L2}  $x \lor (y \lor z) \approx (x \lor y) \lor z$ \qquad $x \land (y \land z) \approx (x \land y) \land z$. 
\item \label{Itm:L3}  $x \lor x \approx x$ \qquad $x \land x \approx  x$. 
\item \label{Itm:L4}  $x \lor (x \land y) \approx x$ \qquad $x \land (x \lor y) \approx x$. 
\item \label{Itm:L5}  $x \land \bot \approx \bot$ \qquad $x \lor \top \approx \top$. 
\item \label{Itm:L6}  $x \land (y \lor z ) \approx (x \land y) \lor (x \land z)$. 
\end{enumerate}
\end{definition}

 \begin{definition}[\cite{Ce99,Ce07}] 
\label{def:dn}
A \emph{
distributive lattice with negation} is an algebra 
$\Al = \la A; \land, \lor, \neg,\bot,\top \ra$ of type $\la 2,2,1,0,0 \ra$
such that $\la A; \land, \lor, \bot, \top \ra$ is a bounded distributive lattice  (Definition~\ref{def:lat})
and the following equations are satisfied: 
\begin{enumerate}[(N1)]
\item \label{Itm:N1} $\nnot \bot\approx \top$. 
\item $\nnot (x \lor y) \approx  \nnot x \land \nnot y$.
\end{enumerate}
%
\end{definition}
We shall denote by $\DN$ the variety of distributive lattices with negation,
and by $\DNE$ the set of equations  axiomatizing this class according to Definition~\ref{def:dn}.

The choice of the class of distributive lattices with negation
as our base variety is due to the following reasons. 
On the one hand,  $\DN$ is sufficiently general to include many algebras of non-classical logics that interest us,
in particular pseudo-complemented distributive lattices and semi-De Morgan algebras (our original case study).
On the other hand, the two items of~Definition~\ref{def:dn} are  some minimal equational requirements  ensuring that the connective $\nnot$
indeed behaves like a negation (in particular, $\nnot$ is order-reversing); also, the theory of
$\DN$ is sufficiently well developed to allow us to rely on a few algebraic lemmas. 
Besides $\DN$, we shall be mainly working with the subvarieties introduced below.

 \begin{definition}[\cite{Sa87}]
\label{def:sdmock}
A 
distributive lattice with negation 
$\Al = \la A; \land, \lor, \neg,\bot,\top \ra$ is: 
\begin{itemize}
\item a \emph{semi-De Morgan algebra}, if $\A$ satisfies the following equations:
\begin{enumerate}[(SDM1)]
\item \label{Itm:SD1} $\nnot \top \approx \bot$. 
\item \label{Itm:SD2} $\nnot \nnot (x \land y) \approx  \nnot \nnot x \land \nnot \nnot y$. 
\item \label{Itm:SD3} $\nnot x \approx \nnot \nnot \nnot x$. 
\end{enumerate}

\item a \emph{De Morgan algebra}, if $\A$ is a semi-De Morgan algebra satisfying:
\begin{enumerate}[(DM)]
\item \label{Itm:DM} $\nnot \nnot x \approx  x$. 
\end{enumerate}

\item a \emph{pseudo-complemented distributive lattice} (\emph{$p$-lattice}, for short), if $\A$ is a semi-De Morgan algebra satisfying:
\begin{enumerate}[(PL)]
\item \label{Itm:PL} $x \land \nnot (x \land y) \approx x \land \nnot y $. 
\end{enumerate}

\item an \emph{Ockham algebra}, if $\A$ satisfies (SDM\ref{Itm:SD1}) plus the following equation:
\begin{enumerate}[(O)]
\item \label{Itm:O1} $\nnot (x \land y) \approx  \nnot x \lor \nnot y$. 
\end{enumerate}
\end{itemize}
\end{definition}
We shall also be interested in the so-called \emph{Berman varieties} of Ockham algebras~\cite{Ber77},
defined via the following terms. Let $\nnot^0 x := x$ and
$\nnot^{n+1} x := \nnot \nnot^{n} x$. For $m \geq 1$ and $n \geq 0$, the variety
$\OC_n^m$ is defined as the subclass of those Ockham algebras that satisfy
the equation $\nnot^{2m+n} x \approx \nnot^n x$. The class of Boolean algebras, viewed as a subvariety of $\DN$, will be denoted by $\BA$; also recall from the preceding Section that $\SDM$, $\DM$ and $\PDL$ denote, respectively,
the variety of semi-De Morgan algebras, De Morgan algebras and  $p$-lattices.
The following inclusions (all proper) hold among the above-defined varieties:
$\BA \subseteq \DM  \subseteq \SDM \subseteq \DN$, 
$\BA \subseteq \DM  \subseteq \OC_n^m \subseteq \OC \subseteq \DN$ and
$\BA \subseteq \PDL \subseteq \SDM \subseteq \DN$.

\begin{figure}[!h]
\centering

\adjustbox{scale=1,center, 
}{
\begin{tikzcd} 
	& {\Large\mathbb{DN}} \\
	{\mathbb{O}} && {\mathbb{SDM}} \\
	{\mathbb{O}^m_n} \\
	{\mathbb{DM}} && {\mathbb{PL}} \\
	& {\mathbb{B}}
	\arrow[from=2-1, to=1-2, no head]
	\arrow[from=1-2, to=2-3, no head]
	\arrow[from=2-1, to=3-1, no head]
	\arrow[from=2-3, to=4-3, no head]
	\arrow[from=3-1, to=4-1, no head]
	\arrow[from=2-3, to=4-1, no head]
	\arrow[from=4-3, to=5-2, no head]
	\arrow[from=4-1, to=5-2, no head]
\end{tikzcd}
}
\caption{Varieties of distributive lattices with negation, ordered by inclusion.} \label{fig:varieties}
\end{figure}
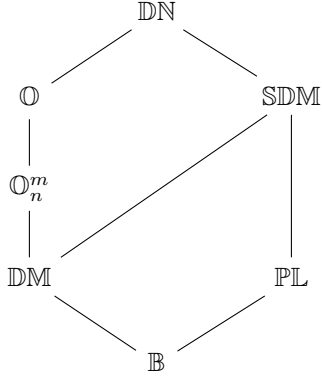

Since its introduction about 
 three decades ago~\cite{Sa87}, semi-De Morgan algebras have been studied
especially in the setting of universal algebra~\cite{Palma2003}
and duality theory~\cite{Ho96,Ce99,Ce07}. On the other hand, 
a logic associated to semi-De Morgan algebras (here denoted $ \vdash^{\leq}_{\SDM} $) has been first considered
in the recent paper~\cite{SDMWollic}. 
Having been introduced in the late 1970's, Ockham lattices are slightly older than semi-De Morgan algebras;
logics associated to (Berman subvarieties of) Ockham lattices are considered in~\cite{MaLin18,MaLin20}.

 De Morgan algebras (i.e.~involutive semi-De Morgan algebras) are worth mentioning in the present context
especially because of their logical interpretation.
In fact, since the 1970's with the seminal papers by N.~Belnap~\cite{Be76,Be77},
the variety $\DM$ 
has been associated to and studied as the standard semantics of 
the Belnap-Dunn four-valued logic (see e.g.~\cite{F97}).
Indeed, 
the consequence relation
$\vdash^{\leq}_{\DM} $ is precisely the Belnap-Dunn logic 
(on the other hand, $\vdash^{\top}_{\DM} $ is strictly stronger than
$\vdash^{\leq}_{\DM} $).
Sub(quasi)varieties of $\DM$ 
have also been studied from a logical point of view
in the more recent papers~\cite{Ri,Pr,AlPrRi}.
From a technical point of view, we shall also be interested in exploiting the structural relation between semi-De Morgan and De Morgan algebras
stated in Lemma~\ref{lem:fey}.

The study of 
$p$-lattices 
can be traced back to the 1920's with V.~Glivenko's classical work on intuitionistic logic. 
From a logical point of view, the importance of $p$-lattices stems from their relation with intuitionism. 
In fact, it is well known that $p$-lattices are precisely the implication-free subreducts of Heyting algebras:
in logical terms, this entails that  the logic $ \vdash^{\leq}_{\PDL} $, or equivalently 
$ \vdash^{\top}_{\PDL} $ (both defined as in Section~\ref{sec:intro}),
captures the implication-free fragment of intuitionistic logic. 

We end the Section with a few algebraic lemmas that will be used to make sure
that certain rules are sound with respect to particular subclasses of 
$\DN$.

\begin{lemma}
\label{lem:forleibn}
Let $\A 
$ be a  semi-De Morgan algebra 
and $a,b,c \in A$. Then,
\begin{enumerate}[(i)]
\item $\nnot (a \land b) = \nnot (\nnot \nnot a \land b) = \nnot ( a \land \nnot \nnot b) = \nnot ( \nnot \nnot a \land \nnot \nnot b)  $.
\item $\nnot (\nnot (\nnot a \land b) \land c) \leq  \nnot (a  \land c)$.
\end{enumerate}
\end{lemma}

\begin{proof}
(i). See~\cite[Lemma~1.1]{Ce07}.

(ii). Let $a,b,c \in A$. Observe that, by the preceding item, 
$\nnot (a \land b) = \nnot (\nnot \nnot a \land b)$.
Since $\nnot$ is order-reversing, from $\nnot a \land b \leq \nnot a $ we have
 $\nnot \nnot a \land c \leq  \nnot (\nnot a \land b)   \land c  $
 and
 $  \nnot (\nnot (\nnot a \land b)   \land c ) \leq \nnot (\nnot \nnot a \land c ) = \nnot ( a \land c )$.
\end{proof}

Let $\Al = \la A; \land, \lor, \neg, 0, 1  \ra$ be a semi-De Morgan algebra.
Defining $A^* := \{ \nnot a : a \in A \}$ and
$a \lor^* b : = \nnot \nnot  ( a \lor  b)$ for all $a,b \in A^*$,
we consider the algebra
$\Al^* = \la A^*; \land, \lor^*, \nnot, 0, 1  \ra$.
It is easy to show that $A^*$ is indeed closed under the operations
$\{\land, \lor^*, \nnot, 0, 1 \}$. Moreover, we have the following result, which may be viewed as a generalization
of Glivenko's theorem relating Heyting and Boolean algebras.

\begin{lemma}[\cite{Sa87}, Thm.~2.4]
\label{lem:fey}
If $\Al$ is a semi-De Morgan algebra, then
$\Al^*$
is a De Morgan algebra.
\end{lemma}

{


The preceding Lemma is interesting for us because of the following logical consequence.
Let $\phi $ be a  formula in the language of semi-De Morgan logic. 
Define the formula $\phi^*$ recursively as follows:
$$
\phi^* : =
\begin{cases}
  \nnot \nnot \phi &\mbox{ if }\phi \in Var \cup \{ \top \} \\
    \nnot \phi^*_1 &\mbox{ if }\phi = \nnot \phi_1 \\
\phi_1^* \land \phi_2^* & \mbox{ if } \phi = \phi_1 \land \phi_2 \\
\nnot \nnot (  \phi_1^* \lor \phi_2^*)  & \mbox{ if } \phi = \phi_1 \lor \phi_2.
\end{cases}
$$

\begin{lemma}
\label{lem:us}
Let $\la \phi, \psi \ra$ be a rule that is sound w.r.t.~$\vdash^{\leq}_{\DM} $ (i.e.~the Belnap-Dunn logic).
Then $\la \phi^*, \psi^* \ra$ is sound w.r.t.~$\vdash^{\leq}_{\SDM} $.
\end{lemma}

\begin{proof}
By contraposition, assume $\la \phi^*, \psi^* \ra$ is not sound in $\vdash^{\leq}_{\SDM} $.
Then there is a semi-De Morgan algebra $\A$ that witnesses
the failure of the inequality $\phi^* \leq \psi^* $.
It is then easy to check that $\A^*$ (which is a De Morgan algebra, by Lemma~\ref{lem:fey})  witnesses
the failure of $\phi \leq \psi $, contradicting the assumption that 
$\la \phi, \psi \ra$ is sound w.r.t.~the Belnap-Dunn logic.
\end{proof}
}

\subsection{Logics}

 the propositional language 
Here, a \emph{logic}
is a structural (Tarskian)  consequence relation on $\Fm$, that is, a
subset of $\PP(Fm) \times Fm$. Logics will be denoted by $\vdash$ with suitable subscripts, regardless
of the way (syntactical or semantical) they are defined. A logic can, for instance, be defined through a \emph{logical matrix}, 
i.e.~a pair $\Mt = \la \A, D \ra$ where $\A$ is an algebra and $D \subseteq A$ a set of designated elements. 
One sets $\Gamma \vdash_\Mt \phi $ iff  for every homomorphism $h \colon Fm \to A$, we have
$h(\phi) \in D$ whenever $h(\Gamma) \subseteq D$.
Similarly, a class of logical matrices defines a logic by considering the intersection of the logics defined by each member of the class.
Another way is by considering a class of partially ordered algebras $\K$, giving rise to the order-preserving logic $\vdash^{\leq}_{\K} $
defined in the Introduction. Indeed, for a class $\K$ of lattice-ordered algebras,
$\vdash^{\leq}_{\K} $ is the logic defined by the class of all matrices
$\la \A, D \ra$ such that $\A \in \K$ and $D$ is a lattice filter of $\A$.

We shall also be interested in logics defined through Hilbert calculi consisting of a finite or denumerable set of rule schemata. By a \emph{Hilbert rule} we mean a pair $\la \Gamma, \phi \ra$, 
usually denoted $\frac{\Gamma}{\psi}$, where $\Gamma \cup \psi \subseteq Fm$. 
When $\Gamma$ is a singleton (say, $\Gamma = \{ \phi \}$ for some
$\phi \in Fm$), we speak of a \emph{formula-to-formula} rule, usually written $\frac{\phi}{\psi}$.  
We shall
write
$\Tfrac{\varphi}{\psi}$ to denote the ``bidirectional rule'', which is really just
 an abbreviation for the pair of formula-to-formula 
rules
$\{ \frac{\varphi}{\psi}, \frac{\psi}{\phi} \}$.
Every set $\Ru$ of Hilbert rules determines a logic $\vdash_{\Ru}$
in the standard way, and
we write $\Gamma \vdash_{\Ru} \phi$ whenever there is a Hilbert derivation
of $\phi$ from $\Gamma$ that uses the rules in $\Ru$.

Below we state formally a result that will be central to our study of the relation between order-preserving logics
and varieties of distributive lattices with negation. 

Recall that a logic is said to be \emph{non-pseudo-axiomatic}
if the set of its theorems is the set of formulas that are derivable from every formula~\cite[p.~78]{Ja06a}.
Every order-preserving logic $\vdash^{\leq}_{\K} $
considered in the present paper is non-pseudo-axiomatic. Moreover,
since all algebras in $\K \subseteq \DN$ have a lattice reduct,
$\vdash^{\leq}_{\K} $ is \emph{semilattice-based} relative to $\land$ and $\K$~\cite[p.76]{Ja06a}.
Therefore, we can apply~\cite[Thm.~3.7]{Ja06a} to obtain the following.

\begin{theorem}
\label{thm:ramon}
There is a dual isomorphism between
the set of all subvarieties of $\DN$, ordered by inclusion,
and 
the set of logics $\vdash^{\leq}_{\K} $,
ordered by extension. 
The isomorphism is given by $\K \mapsto \, \vdash^{\leq}_{\K}$.
\end{theorem}

In the present paper, we will study the problem of
obtaining, from
a basis $\Eq$ for the  equational theory of $\K \subseteq \DN$, 
a set of rules
that form a basis for the logic
$\vdash^{\leq}_{\K}$; in particular, 
 we shall be interested in conditions 
 ensuring that the  set of rules may be taken to be finite.

 \section{The order-preserving logic of $\DN$
 }
 \label{sec:dn}
 
 In this Section we introduce an infinite Hilbert calculus for  
 the order-preserving logic of the variety $\DN$. 
 Our calculus is obtained by translating the set $\DNE$ of equations  that axiomatize $\DN$
 into a set $\Ru^\DNE$ of bidirectional rules, which we  then suitably enlarge  in order to ensure
 that the corresponding inter-derivability relation is a congruence of $\Fm$. 
 After  showing that the denumerable set $\Ru_{\omega}$ of rules thus
obtained axiomatizes $\vdash^{\leq}_{\DN}$
(Corollary~\ref{cor:compdn}), 
we will proceed to show that $\Ru_{\omega}$ 
cannot be replaced by any finite set. This is the main result of this Section: 
\emph{the order-preserving logic of $\DN$ is not finitely based} (Theorem~\ref{th:dnnotfi}). 
We note that most of the results that we proceed to prove below also hold
for more general classes than $\DN$, and thus for logics weaker than 
$\vdash^{\leq}_{\DN}$ (for instance, Lemma~\ref{cong} only relies on having the set of commutativity rules
$\Ru_{\mathsf{C}}$ defined below, etc.). In view of future research, this suggests
the project of  applying
our techniques
to more general logics/classes of algebras. 
Given a set of equations $\Eq :=\{\varphi_i\approx \psi_i :  i\in I \} \subseteq Fm \times Fm$,
 we define the following set of bidirectional rules:
$$\Ru^\Eq :=\Big\{\Tfrac{\varphi_i}{\psi_i}:i\in I\Big\}
.$$

Note that every rule in $\Ru^\Eq$  is formula-to-formula.

Following standard notation, we  use $x,y,z$ etc.~to denote
 variables  used in  equations and $p,q,r$ etc.~to denote  logical variables.
 For instance, the equations  (L\ref{Itm:L1}) in Definition~\ref{def:lat} give us
 $\Tfrac{p\lor q}{q\lor p}$ 
 and  $\Tfrac{p\land q}{q\land p}$, and so on.

Given a set $ \Ru \subseteq Fm \times Fm$ of formula-to-formula rules, 
let 
$\{q_i:i<\omega\}$ be a set of fresh variables such that 
$\{q_i:i<\omega\}\cap \mathsf{var}(\Ru)=\emptyset$. Define:
\begin{align*}
 \Ru_0 &:= \Ru\\  
 \Ru_{n+1} &: =\Big\{\frac{\varphi\lor q_n}{\psi\lor q_n}:\frac{\varphi}{\psi}\in \Ru_n\Big\}\cup \Big\{\frac{\varphi\land q_n}{\psi\land q_n}:\frac{\varphi}{\psi}\in \Ru_n\Big\}\cup 
\Big\{\frac{\neg \psi}{\neg \varphi}:\frac{\varphi}{\psi}\in \Ru_n\Big\}\\
%
\Ru_\omega&:=\bigcup_{n<\omega}\Ru_n 
\end{align*}
%
%
%
Let us also fix  the set $\Ru_{\mathsf{C}}=\{\frac{p\land q}{q\land p}, \frac{p\lor q}{q\lor p}\}$ 
and  
$\Ru_{\mathsf{F}}=\{\frac{}{\top},\frac{p\,,\,q}{p\land q}, \frac{p}{p \lor q}\}$.
As the notation suggests, the set $\Ru_{\mathsf{C}}$ ensures that the conjunction and disjunction
are commutative, while the rules in $\Ru_{\mathsf{F}}$ say that the designated elements
are (non-empty) lattice filters of the algebraic models of the logic\footnote{Observe that
$\Ru_{\mathsf{C}} \subseteq \Ru^{\DNE}$, so we will not need to worry about adding $\Ru_{\mathsf{C}}$ when dealing
with $\vdash^\leq_{\DN}$ and stronger logics.}.

Recall that a logic $\vdash$ is said to be \emph{selfextensional} if the inter-derivability 
relation $ \dashv \vdash$ is a congruence of the formula algebra $\Fm$.
Obviously, every order-preserving logic $\vdash^{\leq}_{\K}$ is  selfextensional:
thus one needs to ensure that the syntactic counterpart of $\vdash^{\leq}_{\K}$
also enjoys this property.

\begin{lemma}\label{cong}
Let $\Ru \subseteq Fm \times Fm$ be a set of formula-to-formula rules such that 
 $\Ru_{\mathsf{C}} \subseteq \Ru $. Then 
the inter-derivability relation $\dashv\vdash_{\Ru_\omega}$ is a congruence of $\Fm$.
\end{lemma}
\begin{proof}
By construction, 
we have that any $\Ru_\omega$-derivation of 
 $\varphi\vdash_{\Ru_\omega} \psi$ can easily be transformed in derivations showing that
 $\varphi\lor \gamma \vdash_{\Ru_\omega}\psi\lor \gamma$, $\varphi\land \gamma\vdash_{\Ru_\omega} \psi\land \gamma$ and $\neg \psi\vdash_{\Ru_\omega} \neg \varphi$ (cf.~the proof of Lemma~\ref{prop:cong}).
 Hence, if $\varphi_i \dashv\vdash_{\Ru_\omega}\psi_i$ we have $\varphi_1\land\varphi_2\dashv\vdash_{\Ru_\omega}\psi_1\land \psi_2$,
 $\varphi_1\land\varphi_2\dashv\vdash_{\Ru_\omega}\psi_1\land \psi_2$ and $\neg\psi_i\dashv\vdash_{\Ru_\omega}\neg \varphi_i$.  
\end{proof}

%
%
%
%

The following lemma is an immediate consequence of the definition of 
$\Ru^{\Eq}$.
 
\begin{lemma}\label{lem:Req}
Let $\vdash$ be a consequence relation over $\Fm$.
 If $\Ru^{\Eq}\subseteq~ \vdash$ and $\dashv\vdash$ is a congruence of $\Fm$, 
 then the quotient $\Fm/{\dashv\vdash}$ satisfies all the 
 equalities in $\Eq$.
 In particular, if $\mathsf{DN}\subseteq \Eq$, then  
 $\Fm/{\dashv\vdash}$ is a
 distributive lattice with negation (Definition~\ref{def:dn})
 with the order given by $\vdash$.
\end{lemma}
 %

Given a set of equations $\Eq \subseteq Fm \times Fm$,
we denote by $\VV_{\Eq}$ the variety axiomatized by $\Eq$.

\begin{theorem}
\label{thm:compl0}
Let  $\Eq \subseteq Fm \times Fm$ be a set of equations such that
 $\mathsf{DN}\subseteq \Eq$. 
 Then
 $\Ru^{\Eq}_\omega \cup \Ru_{\mathsf{F}}$ 
  axiomatizes $\vdash^\leq_{\VV_{\Eq}}$.
 

\end{theorem}

\begin{proof}
Let $\vdash := \, \vdash_{\Ru^{\Eq}_\omega \cup \Ru_{\mathsf{F}}}$.
It is clear that 
$\vdash \, \subseteq \, \vdash^\leq_{\VV_{\Eq}}$.
To prove completeness, 
assume $\Gamma\not\vdash \varphi$ for some $\Gamma \cup \{\phi\} \subseteq Fm$. 
By Lemma~\ref{cong} and the fact that $R_{\mathsf{C}}\subseteq R^\mathsf{DN} \subseteq R^{\Eq}$, the relation 
$\dashv \vdash_{R^\Eq_\omega}$
is a congruence of $\Fm$, which in this proof we denote  by $\equiv$.
Consider the matrix $\tuple{\Fm/{\equiv},F}$ where 
$F= 
\Gamma^{\vdash}\! /{\equiv}$
(observe that 
$\vdash_{R^\Eq_\omega}\,\subseteq\, \vdash$ implies that $F$ is compatible with $\equiv$).
It follows from $R^{\Eq}\subseteq\, \vdash$ and Lemma~\ref{lem:Req} that 
 $\Fm/{\equiv}$ is in $\VV_{\Eq}$. In particular,  $\Fm/{\equiv}$ is a lattice.
Thus, to show that $F$ is a  non-empty lattice filter, it suffices to use the rules in $\Ru_{\mathsf{F}}$.
To conclude the proof, observe that the canonical projection map $\pi \colon Fm \to Fm/{\equiv}$
is a valuation that satisfies all formulas in $\Gamma$ but not $\phi$.
\end{proof}

\begin{corollary}
\label{cor:compdn}

$\vdash_{\Ru^{\mathsf{DN}}_\omega\cup \Ru_{\mathsf{F}}}=\,\vdash^\leq_{\DN}$.

\end{corollary}

Recall that an \emph{atomic formula} is a propositional variable or a constant belonging to our language ($\bot$ or $\top$).

\begin{definition}
The \emph{$\neg$-depth} of an occurrence of an atomic formula $\varphi$ in $\psi$ is the number of $\neg$-headed subformulas of $\psi$ 
with that occurrence of $\phi$. In other words, we consider the tree representation of $\psi$ and a leaf labelled $\phi$ (representing the ocurence of interest) and count the
number of $\neg$-labelled nodes that are ancestors of that leaf. 
The $\neg$-depth of a formula $\psi$ is the maximum $\neg$-depth of the atomic subformulas of $\psi$. 
The $\neg$-depth of a set of rules $\Ru$ is the maximum $\neg$-depth among the formulas in $\Ru$.
We say that a rule $\ru$ is {\em $\neg$-balanced}  if all occurrences of all variables in
$\ru$ have the same the same $\nnot$-depth.
We say that a set of rules $\Ru$ is {\em $\neg$-balanced} if every  rule $ \ru \in \Ru$
is  $\neg$-balanced.
\end{definition}

We shall now focus on invariants of logics axiomatized by $\neg$-balanced rules
having  $\nnot$-depth $k < \omega$. This will allow us to single out certain non-finitely based logics
extending $\vdash^\leq_{\DN}$. To this end, we shall also need the following function.

Let $\{q_{\varphi}:\varphi\in Fm\}$ be a fresh set of variables.
For all $k<\omega$ and for all $\varphi,\psi\in Fm$, let $f_{k}: Fm \to  Fm$ be given by: 
\begin{align*}
f_{k}(\top)&: =\top\\
f_{k}(\bot)&:=\bot\\
 f_{k}(p)&:=p\\
 f_{k}(\varphi\land \psi)&:=f_{k}(\varphi)\land f_{k}(\psi)\\
f_{k}(\varphi\lor \psi)&:=f_{k}(\varphi)\lor f_{k}(\psi)\\
 f_{k}(\neg \varphi)&:= 
\begin{cases}
q_{\neg \varphi} &\mbox{ if } k=0\\
 \neg f_{k-1}(\varphi) &\mbox{ if } k>0
\end{cases}
\end{align*}
Essentially,  $f_k$  replaces in $\phi$ every subformula $\psi$ occurring at  $\neg$-depth $k$ whose main connective is $\neg$ by a fresh  variable indexed by $\psi$.
We extend $f_k$ to sets of formulas, rules 
and sets of rules in the expected way: $f_k(\Gamma)=\{f_k(\varphi):\varphi\in \Gamma\}$, $f_k(\frac{\Gamma}{\varphi})=\frac{f_k(\Gamma)}{f_k(\varphi)}$ and 
$f_k(\mathcal{R})=\{f(\Ru): \ru \in \mathcal{R}\}$.  

\begin{lemma}\label{zip}
Let $\Ru$ be a  set of rules that is
$\neg$-balanced and has $\nnot$-depth $k$.
 Then $\Gamma\vdash_{\Ru} \varphi$ implies $f_n(\Gamma)\vdash_{\Ru} f_n(\varphi)$ for every $n>k$.
  
\end{lemma}
\begin{proof}
Since $n > k$, for each rule  $\frac{\Delta}{\psi}\in \Ru$, we have  $f_n(\Delta)=\Delta$ and  $f_n(\psi)=\psi$. 
Further, for every substitution $\sigma$ (i.e.~for every endomorphism $\sigma \colon Fm \to Fm$)
 there is a substitution 
 $\sigma'$ such that $f_n(\Delta^{\sigma})=f_n(\Delta)^{\sigma'}=\Delta^{\sigma'}$ and $f_n(\psi^\sigma)=f_n(\psi)^{\sigma'}=\psi^{\sigma'}$, where
 $\sigma'(p)=f_{n-j}(\sigma(p))$ and $j$ is the $\nnot$-depth of $p$ in $\frac{\Delta}{\psi}$
 (note that $\sigma'$ is well defined because $\Ru$ is $\nnot$-balanced).
 It is then easy to see (cf.~the proof of Lemma~\ref{prop:cong})  that every $\Ru$-derivation of $\phi$ from $\Gamma$ can be transformed into a derivation of 
 $f_n(\phi)$ from $f_n(\Gamma)$.
\end{proof}


\begin{lemma}\label{zip2}
Let $\Ru\subseteq Fm \times Fm$ 
be $\neg$-balanced and having $\nnot$-depth $k$. 
If $f_{n+k}(\Ru_{n})\not \subseteq \, \vdash_{\Ru_\omega\cup \Ru_{\mathsf{F}}}$
for every $n<\omega$, 
 then the logic $\vdash_{\Ru_\omega\cup \Ru_{\mathsf{F}}}$ is not finitely based.
\end{lemma}
\begin{proof}
Let $\vdash_n=\vdash_{\Ru_n\cup \Ru_{\mathsf{F}}}$ and $\vdash_\omega=\vdash_{\Ru_\omega\cup \Ru_{\mathsf{F}}}$.
As
 $\Ru_\omega=\bigcup_{n<\omega}\Ru_n$, 
 it is enough to show that $\vdash_n\subsetneq \vdash_{n+1}$.
 It is clear that $\Ru_n\cup R_{\mathsf{F}}$ is $\neg$-balanced and with $\neg$-depth $n+k$. 
 Hence, by Lemma~\ref{zip}, $\Gamma\vdash_n\varphi$ iff $f_{n+k}(\Gamma)\vdash_n f_{n+k}(\varphi)$.
 Thus, from  $f_{n+k}(\Ru_{n+1})\not\subseteq\, \vdash_\omega$ and $\vdash_n\subseteq \vdash_\omega$ 
 we conclude that 
 $\Ru_{n+1}\subseteq\, \vdash_{n+1}$ 
 but
 $\Ru_{n+1}\not \subseteq\,  \vdash_n$,
as was required to prove.
\end{proof}

\begin{theorem}
\label{th:dnnotfi}
 The logic 
 $\vdash^\leq_{\DN}$ of distributive lattices with negation
 is not finitely based. 
\end{theorem}

\begin{proof}
Recall that
$\vdash_{\Ru^{\mathsf{DN}}_\omega\cup \Ru_{\mathsf{F}}}=\,\vdash^\leq_{\DN}$
by Corollary~\ref{cor:compdn}.
%
 Then, the result follows directly from Lemma~\ref{zip2}. 
 Indeed, $\Ru^{\mathsf{DN}}$ is $\neg$-balanced and has $\neg$-depth $1$. Moreover, for every
 $n < \omega$, we have
 $\mathsf{r}=\frac{\nnot\nolimits^{n+1}(p\land p))}{\nnot\nolimits^{n+1}p}\in \Ru_n^{\mathsf{DN}}$ and
 $$f_n(\mathsf{r})=\frac{ f_n (\nnot\nolimits^{n+1}(p\land p))}{f_n (\nnot\nolimits^{n+1}p)})=\frac{\neg^n(q_{\neg(p\land p)})}{\neg^n(q_{\neg p})}\notin\,\vdash^\leq_{\DN}.$$
 \end{proof}

Note that the result of Theorem~\ref{th:dnnotfi} holds for
every strengthening of  $\vdash^\leq_{\DN}$ to which Lemma~\ref{zip2} applies.
In particular, let $\Eq$ be a set of equations such that
$\Ru^{\Eq}$ is $\nnot$-balanced and has finite $\neg$-depth.
Then, for Lemma~\ref{zip2} to apply, 
it suffices to have $\BA \subseteq \VV_{\Eq} \subseteq \DN$.
For instance,
denoting by $\vdash^\leq_{\mathbb{O}}$
the logic of order of the variety of Ockham algebras (Definition~\ref{def:sdmock}),
it suffices to check that the rule
$
\Tfrac{\nnot (p\, \land\, q)}{\nnot p \lor \nnot q}
$ is
$\nnot$-balanced
to conclude that $\vdash^\leq_{\mathbb{O}}$ is not finitely based. 
A similar argument shows that,
letting $\K \subseteq \DN$ be the variety of distributive lattices with negation
axiomatized (relatively to $\DN$) by equations
(SDM\ref{Itm:SD1})  and (SDM\ref{Itm:SD2}) from Definition~\ref{def:sdmock},
we have that $\vdash^\leq_{\K}$ is not finitely based.


%
 

\section{The logics of semi-De Morgan algebras and  $p$-lattices}
\label{sec:sdm}

In this Section we show that, unlike $\vdash^\leq_{\DN}$ and $\vdash^\leq_{\mathbb{O}}$, the logic of order 
of semi-De Morgan algebras  $ \vdash^{\leq}_{\SDM} $ is finitely based. In fact, we are going to establish a more general result:
every logic of order extending semi-De Morgan logic is finitely based, provided
the corresponding variety is (Theorem~\ref{th:subva}).




Let $\Ru \subseteq Fm \times Fm$ be a set of rules, and let
 $(\{q\}\cup\{q_i:i<\omega\})\cap \mathsf{var}(\Ru)=\emptyset$.
%
Given a formula $\gamma$, let
%
$g_0(\gamma)=\gamma\land q_0$ and 
$g_{n+1}(\gamma)=\neg g_n (\gamma)\land q_n$.
Given a rule  $\frac{\varphi}{\psi}$ and $n<\omega$, let
$$
\mathsf{r}^{\frac{\varphi}{\psi}}_n=
\begin{cases}
 \frac{g_n(\varphi)\lor q}{g_n(\psi)\lor q} \mbox{ if }n=2k\\[.3cm]
 \frac{g_n(\psi)\lor q}{g_n(\varphi)\lor q}\mbox{ if }n=2k+1
\end{cases}
$$
%
For $n<\omega$, let
$\Ru^g_n=\{\mathsf{r}^{\frac{\varphi}{\psi}}_n:\frac{\varphi}{\psi}\in \Ru\}$,
 $\Ru^g_{\leq n}=\Ru\cup \bigcup_{k\leq n}\Ru^g_k$ 
 and $\Ru^g_\omega=\bigcup_{n<\omega}\Ru^g_{\leq n}$.

\begin{example}
\label{ex:patterng}
Given a rule 
$\frac{\varphi}{\psi}$, 
we have
  $$
 \frac{(\varphi\land q_0)\lor q}{(\psi\land q_0)\lor q} \,\,{\mathsf{r}^{\frac{\varphi}{\psi}}_0} \qquad
 \frac{(\neg(\psi\land q_0)\land q_1)\lor q}{(\neg (\varphi\land q_0)\land q_1)\lor q}\,\, {\mathsf{r}^{\frac{\varphi}{\psi}}_1} \qquad 
 \frac{(\neg(\neg(\varphi\land q_0)\land q_1)\land q_2)\lor q}{(\neg(\neg(\psi\land q_0)\land q_1)\land q_2)\lor q}\,\, {\mathsf{r}^{\frac{\varphi}{\psi}}_2}$$
The general pattern is:
 $$\frac{(\neg\ldots(\neg(\gamma^{\mathsf{up}_n}\land q_0)\land q_1)\ldots\land q_n)\lor q}{(\neg\ldots(\neg(\gamma^{\mathsf{dn}_n}\land q_0)\land q_1)\ldots\land q_n)\lor q}\,\, {\mathsf{r}^{\frac{\varphi}{\psi}}_n} $$ with
 
 $$\gamma^{\mathsf{up}_n}=
\begin{cases}
 \varphi &\text{for even }n \\
  \psi &\text{for odd }n\\
\end{cases} \text{ and }
\gamma^{\mathsf{dn}_n}=
\begin{cases}
 \psi &\text{for even }n \\
  \varphi &\text{for odd }n.
\end{cases} 
$$
\end{example}
Clearly, $\Ru_\omega^g\subseteq \Ru_\omega$,
 where $\Ru_\omega$ is defined as in the previous Section.
%
Let us fix the set $\Ru_{\bullet}$  consisting of the following rules:

$$
\Dfrac{p}{p \lor\bot} {\mathsf{r}^\lor_\bot} \qquad 
\Dfrac{p\lor r}{(p\land\top)\lor r} {\mathsf{r}^\land_\top} \qquad
 \Dfrac{\neg (p\land\top)\lor r}{\neg p\lor r}  {\mathsf{r}^{\neg}_\top}
 $$
 
$$
\Dfrac{(p\lor q)\land r}{(p\land r)\lor (p\land r) }\,\mathsf{r}_\mathsf{dist}^{\lor\land} \qquad
\Dfrac{((p_1\land p_2)\land p_3) \lor q}{(p_1\land (p_2\land p_3)) \lor q}\,\mathsf{r}_{\mathsf{ass}_\land^\lor} \qquad
\Dfrac{\neg(p\lor q)}{\neg p \land \neg q}\,\mathsf{r}_{\mathsf{dm}^{\neg}_\lor}
$$

\bigskip

{ 
\begin{lemma}\label{prop:cong}

If $\Ru_\bullet \subseteq\, \vdash_{\Ru_\omega}$ then
$\vdash_{\Ru^g_\omega\cup \Ru_\bullet}{=}\,\vdash_{\Ru_\omega}$.
\end{lemma}
\begin{proof}
Let $\vdash{=}\vdash_{\Ru^g_\omega}$.
Since $\Ru\subseteq \Ru^g_\omega \,\subseteq \, \vdash_{\Ru_\omega}$, 
it is enough to show that  if
$\varphi \vdash\psi$, given a  fresh variable $q$, we have:
 
 (i) $\varphi\lor q\vdash  \psi\lor q$ \qquad
  (ii) $\varphi\land q\vdash  \psi\land q$ \qquad 
 (iii) $\neg \psi\vdash  \neg \varphi$.

The proof is by induction on the length of the 
derivation 
showing that 
$\varphi\vdash  \psi$.
In the base case we have simply  $\varphi=\psi$, in which case (i), (ii) and (iii) follow immediately.
For the step, assume  $\varphi, \gamma_1, \ldots, \gamma_k, \psi$ is an $\Ru^g_\omega$-derivation and
 by induction hypothesis we have that 
 $\varphi\lor q \vdash \gamma_k\lor q$, 
$\varphi\land q\vdash \gamma_k\land q$ and
$\neg \gamma_k \vdash \neg \varphi$. 
 To conclude the proof, we  consider in each of the cases how to complete the derivations depending on the last rule that was used.
By structurality, it is enough to show that for each rule $\frac{\varphi}{\psi}\in 
\Ru^g_\omega$ we have that (i)--(iii) hold.


Concerning the rules $\frac{\varphi}{\psi}\in \Ru$, we have:
 
\begin{itemize}
 \item[(i)] $\varphi\lor q\vdash_{\mathsf{r}^\land_\top} (\varphi\land \top) \lor q \vdash_{\mathsf{r}^{\frac{\varphi}{\psi}}_0} (\psi\land \top) \lor q  \vdash_{\mathsf{r}^\land_\top} \psi\lor q$
 

  \item[(ii)] $\varphi\land q\vdash_{\mathsf{r}^\lor_\bot} (\varphi\land q) \lor \bot \vdash_{\mathsf{r}^{\frac{\varphi}{\psi}}_0} (\psi\land q) \lor \bot  \vdash_{\mathsf{r}^\lor_\bot} \psi\land q$
 \item[(iii)] $\neg \psi \vdash_{\mathsf{r}^{\neg}_\top} \neg (\psi\land \top) \lor \bot \vdash_{{\mathsf{r}^{\frac{\varphi}{\psi}}_1}} \neg (\varphi\land \top) \lor \bot 
 \vdash_{\mathsf{r}^{\neg}_\top} \neg \varphi \lor \bot \vdash_{\mathsf{r}^{\lor}_\bot} \neg \varphi$.
\end{itemize}

 Now, for each $\frac{\varphi}{\psi}\in \Ru$ and  $j<\omega$,  consider  ${\ru^{\frac{\varphi}{\psi}}_j}=\frac{g_j(\varphi)\lor q'}{g_j(\psi)\lor q'}$. We have:

 

\begin{itemize}
 \item[(i)]  $(g_j(\varphi)\lor q')\lor q \vdash_{\mathsf{r}_\mathsf{ass}^\lor} g_j(\varphi)\lor (q' \lor q) \vdash_{\mathsf{r}^{\frac{\varphi}{\psi}}_{j}} g_j(\psi)\lor (q' \lor q)  \vdash_{\mathsf{r}_\mathsf{ass}^\lor}  (g_j(\psi)\lor q')\lor q $\\
 
  \item[(ii)]  For $j>0$ (the case $j=0$ is analogous) 
\begin{align*}
  (g_j(\varphi)\lor q')\land q&= ((\neg g_{j-1}(\varphi)\land q_j)\lor q')\land q  \\
  &\vdash_{\mathsf{r}_\mathsf{dist}^{\lor\land}} 
    ((\neg g_{j-1}(\varphi)\land q_j)\land q) \lor (q'\land q) \\
    &\vdash_{\mathsf{r}_{\mathsf{ass}_\land^\lor}} 
    ((\neg g_{j-1}(\varphi)\land (q_j\land q)) \lor (q'\land q) \\
    &\vdash_{\mathsf{r}^{\frac{\varphi}{\psi}}_{j}}
    ((\neg g_{j-1}(\psi)\land (q_j\land q)) \lor (q'\land q)\\
    & \vdash_{\mathsf{r}_{\mathsf{ass}_\land^\lor}} ((\neg g_{j-1}(\psi)\land q_j)\land q) \lor (q'\land q)\\
        & \vdash_{\mathsf{r}_\mathsf{dist}^{\lor\land}} (g_j(\psi)\lor q')\land q
\end{align*}
 
  
 \item[(iii)] \qquad $\neg (g_j(\psi)\lor q) \vdash_{\mathsf{r}_{\mathsf{dm}^{\neg}_\lor}}  \neg g_j(\psi)\land \neg q$
\begin{align*}
 &\vdash_{\mathsf{r}^\lor_\bot} 
 (\neg g_j(\psi)\land \neg q)\lor \bot \\
& \vdash_{{\mathsf{r}^{\frac{\varphi}{\psi}}_{j+1}}} 
   (\neg g_j(\varphi)\land \neg q)\lor \bot\\
 &   \vdash_{r^\lor_\bot} 
  \neg (g_j(\varphi)\lor q)\\
   &   \vdash_{\mathsf{r}_{\mathsf{dm}^{\neg}_\lor}} 
  \neg (g_j(\varphi)\lor q)\hspace{3.6cm}
\end{align*}

\end{itemize}
\end{proof}
}

Since $\Ru_{\bullet} \subseteq \, \vdash^\leq_{\DN}$,
by
Lemma~\ref{prop:cong}, we have that 
$(\Ru^{\mathsf{DN}})^g_\omega \cup \Ru_{\bullet} \cup \Ru_{\mathsf{F}}$
provides an alternative (infinite) Hilbert presentation  of $\vdash^\leq_{\DN}$.


In order to obtain a finite axiomatization of $\vdash^\leq_{\SDM}$, let us fix
the set $\Su_{\bullet}$ consisting of  the following rules:


\begin{table}[!h] 
  \begin{tabular}{ccc}
  \\

    \AxiomC{$\neg\neg(p\land q)\lor r$}
    \RightLabel{\,$\ru_\land$
    }
    \UnaryInfC{$\neg\neg p\lor r$}
    \DisplayProof &
       \doubleLine
        \AxiomC{$\neg(\neg\neg p\land q)$}
    \RightLabel{\, $\ru_{\nnot}$ 
    }
    \UnaryInfC{$\neg(p\land q)$}
        \DisplayProof &
    \AxiomC{$\nnot (\nnot p_1 \land p_2)$}
        \AxiomC{$\nnot (\nnot (p_3 \land p_4) \land p_2) $}
    \RightLabel{\, $\ru^{\nnot}_{\land}$ 
    }
      \BinaryInfC{$\nnot (\nnot (p_1 \land p_4) \land p_2)$}
    \DisplayProof \\ \\ 
    
  \end{tabular}   
\end{table}




\begin{proposition}\label{prop++}

Let $\Ru $ be a set of rules,  and let $\Ru_{+}:=\Ru^g_{\leq 2}\cup \Su_\bullet$.
If $\Su_\bullet\subseteq\, \vdash_{\Ru^g_\omega}$,
then $\vdash_{\Ru_{+}}{=}\,\vdash_{\Ru^g_\omega}$.
%
\end{proposition}

\begin{proof}
We just need to show that 
$\mathsf{r}^{\frac{\varphi}{\psi}}_n \in\, \vdash_{\Ru_{+}}$ for $n>2$.
Since $\mathsf{r}^{\frac{\varphi}{\psi}}_1,\mathsf{r}^{\frac{\varphi}{\psi}}_2\in \Ru_{+}$, it suffices to show 
that 
we can derive 
$\Ru^{\frac{\varphi}{\psi}}_{2n+3},\Ru^{\frac{\varphi}{\psi}}_{2n+4}$ using the rules
in $\Ru_{n+}:=\Ru_{+}\cup\{\mathsf{r}^{\frac{\varphi}{\psi}}_{2n+1},\mathsf{r}^{\frac{\varphi}{\psi}}_{2n+2}\}$.

That is, we need to show that, for every $\frac{\varphi}{\psi}\in \Ru$ and $n<\omega$, 
\begin{enumerate}[(i)]
\item 
$g_{2n+3}(\psi)\lor r \vdash_{\Ru_{+n}} g_{2n+3}(\varphi) \lor r$.

We have: 
\begin{align*}
 \gamma_0&=g_{2n+3}(\psi)\lor r\\
&= \neg(g_{2n+2}(\psi)\land q_{2n+3})\lor r \\
 &= 
  \neg( \neg(g_{2n+1}(\psi)\land q_{2n+2})\land q_{2n+3})\lor r\\
  & =
    \neg( \neg(\neg(g_{2n}(\psi)\land q_{2n+1})\land q_{2n+2})\land q_{2n+3})\lor r \\
    &\vdash_{\mathsf{r}_\land}  \neg( \neg\neg(g_{2n}(A)\land q_{2n+1})\land q_{2n+3})\lor r \\
    &\vdash_{\mathsf{r}_{\neg}}  \neg( g_{2n}(\psi)\land (q_{2n+1}\land q_{2n+3}))\lor r \\
        &\vdash_{\mathsf{r}_{2n+1}^{\frac{\varphi}{\psi}}} \neg( g_{2n}(\varphi)\land (q_{2n+1}\land q_{2n+3}))\lor r \\
       &\vdash_{\mathsf{r}_{\neg}}  \  \neg( \neg\neg(g_{2n}(\varphi)\land q_{2n+1})\land q_{2n+3})=\gamma_1
\end{align*}

 Further, 
   $$\gamma_0,\gamma_1\vdash_{\ru^{\nnot}_{\land}} \neg( \neg(\neg(g_{2n}(\varphi)\land q_{2n+1})\land q_{2n+2})\land q_{2n+3})\lor r = g_{2n+3}(\varphi)\lor r.$$
   
   Hence, $g_{2n+3}(\psi)\lor r \vdash_{\Ru_{n+}} g_{2n+3}(\varphi) \lor r$.
   
  
  \item
  $g_{2n+4}(\varphi)\lor r \vdash_{\Ru_{+n}} g_{2n+4}(\psi) \lor r$. 
%

We have:
\begin{align*}
 \gamma_0&=g_{2n+4}(\varphi)\lor r\\
 &= \neg(g_{2n+3}(\varphi)\land q_{2n+4})\lor r\\
 & = 
  \neg( \neg(g_{2n+2}(\varphi)\land q_{2n+3})\land q_{2n+4})\lor r\\
   &=
    \neg( \neg(\neg(g_{2n+1}(\varphi)\land q_{2n+2})\land q_{2n+3})\land q_{2n+4})\lor r \\
      &\vdash_{\mathsf{r}_\land}  \neg( \neg\neg(g_{2n+1}(\varphi)\land q_{2n+2})\land q_{2n+4})\lor r  \\
    &\vdash_{\mathsf{r}_{\neg}}  \neg( g_{2n+1}(\psi)\land (q_{2n+2}\land q_{2n+4}))\lor r=\varphi_2 \\
        &\vdash_{\mathsf{r}_{2n+2}^{\frac{\varphi}{\psi}}}  \neg( g_{2n+1}(\psi)\land (q_{2n+2}\land q_{2n+4}))\lor r  \\
       &\vdash_{\mathsf{r}_{\neg}}  \neg( \neg\neg(g_{2n+1}(\psi)\land q_{2n+2})\land q_{2n+4})=\gamma_1
\end{align*}

 Further, 
   $$\gamma_0,\gamma_1\vdash_{\ru^{\nnot}_{\land}} \neg( \neg(\neg(g_{2n+1}(\psi)\land q_{2n+2})\land q_{2n+3})\land q_{2n+4})\lor r = g_{2n+4}(\psi)\lor r.$$
   
   Hence, $g_{2n+4}(\varphi)\lor r \vdash_{\Ru_{n+}} g_{2n+4}(\psi) \lor r$.


\end{enumerate}

\end{proof}

\begin{lemma}
\label{lem:regsnd}
The rule 
$\ru^{\nnot}_{\land}$
 is sound in $ \vdash^{\leq}_{\SDM} $. 
\end{lemma}


\begin{proof}
Using the semi-De Morgan equations 
(SDM\ref{Itm:SD1})--(SDM\ref{Itm:SD3}), 
it is easy to show that 
the rule
$
\nnot (\nnot p_1 \land p_2) \land
\nnot (\nnot (p_3 \land p_4) \land p_2)
 \vdash
\nnot (\nnot (p_1 \land p_4) \land p_2)
$ 
is sound in $\vdash^{\leq}_{\SDM} $ if and only if 
$
( \nnot \nnot p_1 \lor^*   \nnot p_2) \land
 ( (\nnot \nnot p_3 \land \nnot \nnot p_4) \lor^* \nnot p_2)
 \vdash
 (\nnot \nnot p_1 \land \nnot \nnot p_4) \lor^* \nnot \nnot p_2
$ is sound in $\vdash^{\leq}_{\SDM} $. 
Letting 
$\phi := ( p_1 \lor  \nnot p_2) \land
 ( ( p_3 \land p_4) \lor  \nnot p_2)
$ and
$\psi: =  ( p_1 \land  p_4) \lor p_2$,
the rule $\phi \vdash \psi$ is easily seen to be sound in $\vdash^{\leq}_{\DM} $.
Moreover,
$\phi^* = ( \nnot \nnot p_1 \lor^*   \nnot p_2) \land
 ( (\nnot \nnot p_3 \land \nnot \nnot p_4) \lor^* \nnot p_2)
$
and 
$\psi^* =  (\nnot \nnot p_1 \land \nnot \nnot p_4) \lor^* \nnot \nnot p_2
$.
The soundness of
$\phi^* \vdash \psi^*$ w.r.t.~$\vdash^{\leq}_{\SDM} $ then
 follows from Lemma~\ref{lem:us}.
\end{proof}

\begin{theorem}
\label{thm:compl}
Let $\Eq \subseteq Fm \times Fm$ be a set of equations such that
$\mathsf{SDM} \subseteq \Eq$. 
 Then
 $\Ru^{\Eq}_{+}\cup  \Ru_{\bullet}\cup \Ru_{\mathsf{F}}$  axiomatizes $\vdash^\leq_{\VV_{\Eq}}$.
\end{theorem}
\begin{proof}
Since $\mathsf{DN}\subseteq \mathsf{SDM}\subseteq \Eq$,  we know by Theorem~\ref{thm:compl0}
that
 $\Ru^{\Eq}_\omega \cup \Ru_{\mathsf{F}}$
 axiomatizes $\vdash^\leq_{\VV_\Eq}$. 
From $\Ru_\bullet\,\subseteq\, \vdash^\leq_{\SDM}\,\subseteq\,\vdash^\leq_{\VV_{\Eq}}$ we obtain by Lemma~\ref{prop:cong} that $(\Ru^\Eq)^g_\omega\cup \Ru_\bullet$ axiomatizes $\vdash^\leq_{\VV_{\Eq}}$.
 Moreover, $\Su_\bullet\,\subseteq\, \vdash^\leq_{\SDM}\,\subseteq\,\vdash^\leq_{\VV_{\Eq}}$ (Lemma~\ref{lem:regsnd} deals with the less obvious case). 
Hence, by Proposition~\ref{prop++} we conclude that  $\Ru^{\Eq}_{+}\cup  \Ru_{\bullet}\cup \Ru_{\mathsf{F}}$ axiomatizes $\vdash^\leq_{\VV_{\Eq}}$.
\end{proof}


\begin{example}\label{ex:SDMPL}
By Theorem~\ref{thm:compl}, the set
$
\Ru^{\mathsf{SDM}}_{+}\cup  \Ru_{\bullet}\cup \Ru_{\mathsf{F}}=(\Ru^{\mathsf{SDM}})^g_{\leq 2} \cup \Ru_\bullet\cup \Su_\bullet\cup \Ru_{\mathsf{F}}$ axiomatizes 
$\vdash^{\leq}_{\SDM}$. Since $(\Ru^{\mathsf{SDM}})^g_{\leq 2}=\Ru^{\mathsf{SDM}}\cup (\Ru^{\mathsf{SDM}})^g_{0}\cup (\Ru^{\mathsf{SDM}})^g_{1}\cup(\Ru^{\mathsf{SDM}})^g_{ 2}$,
the 
axiomatization thus obtained  
consists of $4\times 
 | \mathsf{SDM} | + 
 | \Ru_\bullet\cup \Su_\bullet\cup \Ru_{\mathsf{F}} |
 =(4\times 16)+11=75$ rules, many of which  bidirectional.
However, it is not hard to see that $\Ru\subseteq \, \vdash_{(\Ru^{\mathsf{SDM}})^g_0\cup \Ru_{\bullet}}$,
which allows one to reduce the number of rules to
 $(3\times 
 | \Eq |
 )
 +11=59$. 
Further simplifications are of course possible, and in particular cases one may obtain
a much more compact axiomatization. 
A certain amount of redundancy in the set of rules obtained is the price we have to pay
for the generality and modularity of our approach. 
Regarding the latter aspect, observe for instance that
$\PDL$ is axiomatized, relatively to $\SDM$, 
by adding the pseudo-complement equation 
$x \land \nnot (x \land y) \approx x \land \nnot y $.
Adding the rule
$\Tfrac{p \land \nnot (p \land q)}{p \land \nnot q}\,\mathsf{r_P}$ is not sufficient, for
we also need to ensure 
that the resulting logic be self-extensional. 
To achieve this, by Theorem~\ref{thm:compl}, it is enough to add
the three rules: $g_i(\mathsf{r_P})$ for $0\leq i\leq 2$.
We then have that  $\vdash^\leq_\PDL$ is axiomatized over $\vdash^\leq_\SDM$ by
$\{g_i(\mathsf{r_P}):0\leq i\leq 2\}$.
%
%
%
%
%
%
\end{example}

Theorem~\ref{thm:compl} also provides a means to obtain (alternative) finite axiomatizations of other logics of order 
above $\vdash^{\leq}_{\SDM}$. 
 In particular, we can obtain a finite axiomatization of the logic
of $p$-lattices (Definition~\ref{def:sdmock}), i.e.~the implication-free fragment of intuitionistic logic,
that is alternative to the one introduced in~\cite{ReV94}. We provide a general formulation of this
observation below in Theorem~\ref{th:subva}.

\begin{theorem}
\label{th:subva}
Let $\VV \subseteq \SDM$ be a variety. 
The following  are equivalent:
\begin{enumerate}[(i)]
 \item $\VV$ is axiomatized by a finite set of equations.
 \item $\vdash^\leq_{\VV}$ is axiomatized by a finite set of finitary rule schemata.
\end{enumerate}
\end{theorem}

\begin{proof}


That (i) implies (ii) follows directly from Theorem~\ref{thm:compl}. 

For the other direction, assume (ii) holds, so $\vdash^\leq_{\VV}$ is axiomatized by a finite set $\Ru$ of finitary rule schemata.
Given a rule
$\ru = \frac{\Gamma}{\phi}$, let $\Eq(\mathsf{r})$ be the equation $\bigwedge \Gamma\land \phi \approx  \bigwedge\Gamma$. 
Let $\Eq_{\Ru} := \sdm \cup \{ \Eq (\ru) : \ru \in \Ru \}$. Observe that the set 
$\Eq_{\Ru}$ is finite, and let $\VV'$ be the variety defined by the equations $\Eq_{\Ru}$.
We claim that $\VV' = \VV$. Indeed, it is clear that $\VV \subseteq \VV'\subseteq \SDM$ and therefore $\vdash^\leq_\SDM\,\subseteq\,\vdash^\leq_{\VV'}\,\subseteq\,\vdash^\leq_{\VV}$. 
For the other direction, we start by observing that for each $\frac{\phi}{\psi}\in \Ru$ we have
$\Tfrac{\bigwedge \Gamma\land \phi}{ \bigwedge\Gamma}\in \Ru^{\Eq_\Ru}$.
This, together with the fact that 
$\frac{p\,,\, q}{p\land q},\frac{p\land q}{q}\,\in\, \vdash^\leq_{\SDM}\,\subseteq\, \vdash^\leq_{\VV'}$, implies that 
$\Ru\!\subseteq\,\vdash^\leq_{\VV'}$. Hence, $\vdash^\leq_{\VV}\,\subseteq\,\vdash^\leq_{\VV'}$.
%
%
%
From, $\vdash^\leq_{\VV'}=\vdash^\leq_{\VV}$ and Theorem~\ref{thm:ramon} we conclude that $\VV=\VV'$. 
%
%
%
%
%
%
\end{proof}

\section{Order-preserving logics of Berman varieties}
\label{sec:ber}

We have shown in Section~\ref{sec:sdm} how to obtain a finite axiomatization of
the order-preserving logic $\vdash^{\leq}_{\K}$ with   $\K \subseteq \SDM$.
Now, suppose $\K \subseteq \OC$ is a  variety of Ockham algebras.
As observed earlier, 
$\vdash^{\leq}_{\OC}$ is not finitely based. 
However,
if we restrict our attention to a Berman variety $\OC_n^m$ of 
Ockham algebras, then we can adapt the technique employed 
in the preceding Section to obtain a finite Hilbert
axiomatization for $\vdash^{\leq}_{\OC_n^m}$ 
(an infinite one being directly given by
Theorem~\ref{thm:compl0}).
%


From now on, let us fix a variety $\OC^m_n$, with $m,n<\omega$,
and let $\Eq^m_n$ be the equations axiomatizing $\OC^m_n$.
Let $k<\omega$ and let $t$ be a fresh variable. 
Given a rule $\frac{\phi}{\psi}$,
define 
$\mathsf{s}_{2k}^{\frac{\varphi}{\psi}}=\frac{\neg^{2k}(\varphi)\lor t}{\neg^{2k}(\psi)\lor t}$ and $\mathsf{s}_{2k+1}^{\frac{\varphi}{\psi}}=\frac{\neg^{2k+1}(\psi)\lor t}{\neg^{2k+1}(\varphi)\lor t}$. 
Letting:
$$
\frac{p\land q}{p} {\mathsf{r}_\land^1} \qquad 
 \frac{p\land q}{q} {\mathsf{r}_\land^2} \qquad
 \frac{p\,,\,q}{p\land q} {\ru^{\mathsf{in}}_\land}  
 $$
define
$\Ru^{mn}_\land : = \Ru^{\Eq^m_n}\cup \{ 
\mathsf{r}_\land^1, 
\mathsf{r}_\land^2\}$, 
%
%
and  $\mathcal{O}^m_n=\Ru^{mn}_\land\cup \{\mathsf{s}_i(\mathsf{r}):i\leq2m+n,\mathsf{r}\in\Ru^{mn}_\land \}\cup 
\{
\mathsf{r}^{\mathsf{in}}_{\land}\}.$

\begin{lemma}\label{l:mn}
 The relation $\dashv\vdash_{\mathcal{O}^m_n}$ is a congruence of $\Fm$.
\end{lemma}
\begin{proof}
The key difference with the cases considered in the previous Sections 
is that $\Tfrac{\neg (p\land q)}{\neg p\lor \neg q}\,\mathsf{r}_\mathsf{dm}^{\land\lor}\in\Ru^{mn}_\land$, 
 but recall that the following rules are also in $\Ru^{mn}_\land$: 

$$
 \Tfrac{(p\land q)\land r}{p\land (q\land r)}\,\mathsf{r}_\mathsf{ass}^\land\qquad  \Tfrac{(p\lor q)\lor r}{p\lor (q\lor r)}\,\mathsf{r}_\mathsf{ass}^\lor 
%
\qquad \Tfrac{\neg^{2m+n}p}{\neg^{n}p} {\mathsf{r}^m_n}\qquad
\Tfrac{p\lor (q\land r)}{(p\lor q)\land (p\lor r)}\,\mathsf{r}_\mathsf{dist}^{\lor\land}$$

 In the presence of $\mathsf{r}_\mathsf{dm}^{\land\lor}$, we can show directly that, if $\Gamma\vdash \varphi$, then
  $$\text{(i)} \Gamma\lor u\vdash \varphi\lor u \qquad \text{(ii)} \Gamma\land u\vdash \varphi\land u \qquad \text{(iii)} \neg \varphi\vdash \bigvee_{\gamma\in \Gamma} \neg \gamma$$
  where $u$ is a fresh variable. 
Once more it is enough to show that (i)--(iii)
are satisfied when 
 $\mathsf{r}=\frac{\Gamma}{\varphi}\in\mathcal{O}^m_n$.
 
 For $\mathsf{r}
 =\frac{p\,,\, q}{p\land q}$, we have:
 
 \begin{itemize}
 \item[(i)] $p\lor u,q\lor u\vdash_\mathsf{r} (p\lor u)\land (q\lor u) \vdash_{\mathsf{r}_\mathsf{dist}^{\lor\land} } (p\land q)\lor u$
 
   \item[(ii)] $p\land u,q\land u\vdash_\mathsf{r} (p\land u)\land (q\land u) \vdash_{\mathsf{r}_\mathsf{ass}^\land} (p\land q)\land u$
 \item[(iii)] $\neg (p\land q) \vdash_{\mathsf{r}_\mathsf{dm}^{\land\lor}} \neg p \lor \neg q$.
\end{itemize} 
 
 For  
  $\mathsf{r}=\frac{\varphi}{\psi}\in \Ru^{mn}_\land$, we have: 
 
\begin{itemize}
 \item[(i)] $\varphi\lor u \vdash_{\mathsf{s}^{\frac{\varphi}{\psi}}_0}  \psi\lor u$
  \item[(ii)] $\varphi\land u\vdash_{\mathsf{r}_\land^j} \varphi,u \vdash_{\mathsf{r}} \psi,u \vdash_{\mathsf{r}^{\mathsf{in}}_\land} \psi\land u$
 \item[(iii)] $\neg \psi \vdash_{{\mathsf{r}^\lor_\bot}} \neg \psi\lor \bot  \vdash_{\mathsf{s}^{\frac{\varphi}{\psi}}_1} \neg \varphi \lor \bot \vdash_{\mathsf{r}^\lor_\bot} \neg \varphi$.
\end{itemize}

For 
$\mathsf{s}_{k}^{\frac{\varphi}{\psi}}\in  \{\mathsf{s}_i(\mathsf{r}):i\leq2m+n,\mathsf{r}\in\Ru^{mn}_\land \}$, 
   %
we let 
$\gamma^\mathsf{up}_k=\varphi$ and $\gamma^\mathsf{dn}_k=\psi$ if $k$ is odd, and $\gamma^\mathsf{up}_k=\varphi$ and $\gamma^\mathsf{dn}_k=\psi$  if $k$ is even. We can write $\mathsf{s}_{k}^{\frac{\varphi}{\psi}}=\frac{\neg^{k}(\gamma^\mathsf{up}_k)\lor q}{\neg^{k}(\gamma^\mathsf{dn}_k)\lor q}$, as in Example~\ref{ex:patterng}.


 \begin{itemize}
 \item[(i)] $(\neg^{k}(\gamma^\mathsf{up}_k)\lor q)\lor r \vdash_{\mathsf{r}_\mathsf{ass}^\lor} \neg^{k}(\gamma^\mathsf{up}_k)\lor (q\lor r ) \vdash_{\mathsf{s}^{\frac{\varphi}{\psi}}_{k}}   \neg^{k}(\gamma^\mathsf{dn}_k)\lor (q\lor r ) \vdash (\neg^{k}(\gamma^\mathsf{dn}_k)\lor q)\lor r $
  \item[(ii)] $(\neg^{k}(\gamma^\mathsf{up}_k)\lor q)\land r \vdash_{\mathsf{r}_\land^j}(\neg^{k}(\gamma^\mathsf{up}_k)\lor q), r \vdash_{\mathsf{s}^{\frac{\varphi}{\psi}}_{k}} (\neg^{k}(\gamma^\mathsf{dn}_k)\lor q), r \vdash (\neg^{k}(\gamma^\mathsf{dn}_k)\lor q)\land r $
 \item[(iii)] To show that
 $\neg(\neg^{k}(\gamma^\mathsf{dn}_k)\lor q)\vdash 
 \neg(\neg^{k}(\gamma^\mathsf{up}_n)\lor q)
 $ we must consider two cases.

 If $k+1\leq 2m+n$, then 
 $$\neg(\neg^{k}(\gamma^\mathsf{dn}_k)\lor q)\vdash_{\mathsf{r}_\mathsf{dm}^{\land\lor}}\neg^{k+1}(\gamma^\mathsf{dn}_k)\lor \neg q\vdash_{\mathsf{s}_{k+1}^{\frac{\varphi}{\psi}}} \neg^{k+1}(\gamma^\mathsf{up}_k)\lor \neg q\vdash_{\mathsf{r}_\mathsf{dm}^{\land\lor}} \neg(\neg^{k}(\gamma^\mathsf{dn}_k)\lor q).$$
  Otherwise, let $k+1=n+(i2m+j)$ for $i>0$ and $0\leq j< 2m$ (and thus $n+j< 2m+n$).
 
 
 We have: 
\begin{align*}
 \neg(\neg^{k}(\gamma^\mathsf{dn}_k)\lor q)&\vdash_{\mathsf{r}_\mathsf{dm}^{\land\lor}} 
 \neg\nolimits^{k+1}(\gamma^\mathsf{dn}_k)\land \neg q\\
 &\vdash_{\mathsf{r}^m_n} \neg\nolimits^{n+j}(\gamma^\mathsf{dn}_k)\land \neg q\\
 &\vdash_{\mathsf{s}_{k+1}^{\frac{\varphi}{\psi}}}  \neg\nolimits^{n+j}(\gamma^\mathsf{up}_k)\land \neg q\\
 &\vdash_{\mathsf{r}^m_n} \neg\nolimits^{k+1}(\gamma^\mathsf{up}_k)\land \neg q \\
& \vdash_{\mathsf{r}_\mathsf{dm}^{\land\lor}}  \neg(\neg\nolimits^{k}(\gamma^\mathsf{dn}_k)\lor q)
\end{align*}
 %
 \end{itemize}

\end{proof}


\begin{theorem}
The set of rules $\mathcal{O}^m_n\cup \{\frac{}{\top}\}$ axiomatizes 
$\vdash^\leq_{\OC^m_n}$ 
\end{theorem}
\begin{proof}
It is clear that $\mathcal{O}^m_n\cup \{\frac{}{\top}\}\subseteq\, \vdash^\leq_{\OC^m_n}$. 
Completeness follows by a similar reasoning as in Theorem~\ref{thm:compl0} from Lemma~\ref{l:mn}.
\end{proof}
We note that, given $\Eq$ such that $\VV^\Eq\subseteq \OC^m_n$, it is easy to see that
$\vdash^\leq_{\VV^\Eq}$ is axiomatized, relatively to  $\vdash^\leq_{\OC^m_n}$,
 by the set $\Ru^{\Eq} \cup \{\mathsf{s}_i(\mathsf{r}):i\leq2m+n,\mathsf{r}\in\Ru^{\Eq} \}$. 
Hence, if $\Eq$ is finite then $\vdash^\leq_{\VV^\Eq}$ is finitely based.
In fact, one can easily adapt the argument of Theorem~\ref{th:subva} to obtain the following:

\begin{corollary}
\label{cor:subvaoc}
Let $\VV \subseteq \OC^m_n$ be a variety. 
The following  are equivalent:
\begin{enumerate}[(i)]
 \item $\VV$ is axiomatized by a finite set of equations.
 \item $\vdash^\leq_{\VV}$ is axiomatized by a finite set of finitary rule schemata.
\end{enumerate}
\end{corollary}



\section{
$\top$-assertional logics}
\label{sec:top}

As mentioned earlier, another logic 
(alternative to $\vdash^{\leq}_{\K} $)
canonically associated to a given class $\K$ 
of 
algebras 
having a constant $\top$
is the so-called $\top$-assertional logic $\vdash^{\top}_{\K} $ determined
by the class of all matrices $\{ \la \A, \{ \top \} \ra  :  \A \in \K
\}$. By definition, $\vdash^{\top}_{\K} $ is stronger than
$\vdash^{\leq}_{\K} $,
but it is well known that 
$\vdash^{\top}_{\K} = \, \vdash^{\leq}_{\K} $
for
$\K = \BA$ or  $\K = \PDL$.
On the other hand, it is easy to check that 
$\vdash^{\top}_{\DN} \neq \, \vdash^{\leq}_{\DN} $.
For this, it suffices to observe that the rule
$$
\frac{p \land \neg p}{\neg q} \, \ru_{\mathsf{wxc}}
$$
is sound w.r.t.~$\vdash^{\top}_{\DN}$
but not w.r.t.~$ \vdash^{\leq}_{\DN} $. 
The same example witnesses
$\vdash^{\top}_{\SDM} \neq \, \vdash^{\leq}_{\SDM} $
and 
$\vdash^{\top}_{\OC} \neq \, \vdash^{\leq}_{\OC} $.

In this Section we take a closer look at the assertional logic
$\vdash^{\top}_{\SDM}$ 
from an algebraic logic point of view. This perspective will allow us to obtain further information
on the poset of finitary selfextensional extensions of $\vdash^{\leq}_{\SDM} $, as well as to 
provide  a Hilbert calculus for  $\vdash^{\top}_{\SDM}$.
For all unexplained terminology used in this Section, we refer the reader to~\cite{FJa09}.

As mentioned in the Introduction, all logics considered in this paper
are non-protoalgebraic. We state this formally below. 

\begin{theorem}
\label{thm:nonprot}
Let $\K \subseteq \DN$.
If $\PDL \subseteq \K $ or $\DM \subseteq \K $,
then 
$\vdash^{\top}_{\K}$ (and, a fortiori, $\vdash^{\leq}_{\K}$)
is not protoalgebraic.
\end{theorem}

\begin{proof}
Observe that both 
$\vdash^{\top}_{\PDL}$
and
$\vdash^{\top}_{\DM}$
are non-protoalgebraic. The former was
remarked in~\cite[p.~320]{ReV93b},
while the latter is proved in~\cite[Thm.~5.1]{AlPrRi}.
The result then follows from the observation that the property of being protoalgebraic
is preserved by extensions. 
(Indeed, we notice that~\cite[Thm.~5.1]{AlPrRi} even entails that $\vdash^{\top}_{\K}$
is not protoalgebraic  for  every $\K$ with  $\BA \varsubsetneq \K \subseteq \DM$.)
\end{proof}

We next provide a better description of reduced matrix models of 
$\vdash^{\leq}_{\SDM} $.
Recall that a matrix $\Mt$ is  a \emph{model} of a logic $\vdash$ when
$ \vdash \, \subseteq \,  \vdash_{\Mt}$.
The \emph{Leibniz congruence} $\Leibniz_{ \alg{A}} ( D)$ of a 
matrix $\Mt = \la \A, D \ra$ is the largest congruence of $\A$ that is compatible with 
$D$ in the following sense: for all $a,b \in A$, if $a \in D$ and $\la a,b \ra \in \Leibniz_{ \alg{A}} ( D)$,
then $b \in D$.  A matrix $\Mt = \la \A, D \ra$ is \emph{reduced} when $\Leibniz_{ \alg{A}} ( D)$ is the identity relation.

\begin{proposition}
\label{prop:leibn}
Let  $\Mt = \la \A, D \ra$ be a model of $\vdash^{\leq}_{\SDM} $ with $\A \in \SDM$, and let $a,b \in A$. Then
$\la a,b \ra \in \Leibniz_{ \alg{A}} ( D) $ if and only if, for all
$c_1, c_2, c_2 \in A$, the following conditions hold:
\begin{enumerate}[(i)]
\item $a \lor c_1 \in D$ \ iff \ $b \lor c_1 \in D$,
\item $\nnot (a \land c_2) \lor c_1 \in D$ \ iff \ $\nnot (b \land c_2) \lor c_1 \in D$,
\item $ \nnot (\nnot (a \land c_3) \land c_2) \lor c_1 \in D$ \ iff \ $ \nnot (\nnot (b \land c_3) \land c_2) \lor c_1 \in D$.
\end{enumerate}
\end{proposition}

\begin{proof}
Let $\theta$ be the relation defined by items (i)--(iii). Let us check that $\theta$ is compatible with the
algebraic operations of $\Al$.

$(\nnot)$. Assume $\la a , b \ra \in \theta$.
That  $\nnot a \lor c_1 \in D$ \ iff \ $\nnot b \lor c_1 \in D$ follows from (ii):
observe that, taking $c_2 = \top $, we have 
$\nnot a \lor c_1 = \nnot (a \land \top) \lor c_1$
and 
$\nnot b \lor c_1 = \nnot (b \land \top) \lor c_1$.
A similar reasoning,
taking $c_3 = \top$ in (iii), shows that  
$\nnot ( \nnot a \land c_2) \lor c_1 \in D$ \ iff \ $\nnot (\nnot b \land c_2) \lor c_1 \in D$.
Now, assume 
$ \nnot (\nnot (\nnot a \land c_3) \land c_2) \lor c_1 \in D$.
By Lemma~\ref{lem:forleibn}.ii, we have 
$ \nnot (\nnot (\nnot a \land c_3) \land c_2) \lor c_1 
\leq  \nnot (\nnot \nnot a  \land c_2) \lor c_1 
=  \nnot ( a  \land c_2) \lor c_1  $.
Hence, $  \nnot ( a  \land c_2) \lor c_1 \in D$, and we can apply (ii) to obtain
 $  \nnot ( b  \land c_2) \lor c_1 =   \nnot ( \nnot \nnot b  \land c_2) \lor c_1 \in D$.
Thus we have
 $ (\nnot ( \nnot \nnot b  \land c_2) \lor c_1) \land ( \nnot (\nnot (\nnot a \land c_3) \land c_2) \lor c_1 ) \in D$,
 because $D$ is closed under $\land$.
 Then, taking
 $p_1 = \nnot b$, $p_2 = c_2$, $q = c_1$, $p_3 = \nnot a $, $p_4 = c_3$
  in $\ru^{\nnot}_{\land}$,
we have
 $ \nnot (\nnot (\nnot b \land c_3) \land c_2) \lor c_1 \in D$.
 
To check that $\theta$ is compatible with the binary operations,
assume $\la a_1, b_1 \ra, \la a_2, b_2 \ra \in \theta$.
Relying on  completeness (Theorem~\ref{thm:compl}),
we can use any logical rule 
$\frac{\phi}{\psi}$
such that
$\phi \leq \psi$ is an inequality 
valid in $\SDM$.
In particular, in the proof below, by (e.g.) `commutativity' for $\land$ we shall refer not only to the rule
$\frac{p \land q}{q \land p}$,
but also 
$\frac{\nnot (p \land q)}{\nnot (q \land p) }$,
$\frac{\nnot (p \land q) \lor r}{\nnot (q \land p) \lor r }$,
etc. In the computations that follow, we shall skip the steps that follow trivially (by symmetry) from
the preceding ones; the dots (\ldots)  will be used to indicate the  passages that have been omitted. 
 
 $(\land)$. 
 We have:
\begin{enumerate}[(i)]
\item $(a_1 \land a_2) \lor c_1 \in D$
\begin{align*}
& \text{ iff } (a_1 \lor c_1) \land (a_2 \lor c_1) \in D & \text{by distributivity}\\
& \text{ iff } a_1 \lor c_1, a_2 \lor c_1 \in D & \text{by }  \frac{p \land q}{p} \\
& \text{ iff } b_1 \lor c_1, b_2 \lor c_1 \in D & \text{by (i)}\\
(\ldots) 
& \text{ iff } (b_1 \land b_2) \lor c_1 \in D.
\end{align*}
\item $\nnot ( (a_1 \land a_2) \land c_2) \lor c_1 \in D$
\begin{align*}
& \text{ iff } \nnot ( a_1 \land (a_2 \land c_2)) \lor c_1 \in D & \text{by $\land$-associativity}\\
& \text{ iff } \nnot ( b_1 \land (a_2 \land c_2)) \lor c_1 \in D & \text{by (ii)}\\
& \text{ iff } \nnot ( (b_1 \land a_2) \land c_2) \lor c_1 \in D & \text{by $\land$-associativity}\\
& \text{ iff } \nnot ( (a_2 \land b_1) \land c_2) \lor c_1 \in D & \text{by $\land$-commutativity}\\
& \text{ iff } \nnot ( a_2 \land (b_1 \land c_2)) \lor c_1 \in D & \text{by $\land$-associativity}\\
& \text{ iff } \nnot ( b_2 \land (b_1 \land c_2)) \lor c_1 \in D & \text{by (ii)}\\
(\ldots) 
& \text{ iff } \nnot ( (b_1 \land b_2) \land c_2) \lor c_1 \in D.
\end{align*}
\item $\nnot (\nnot ((a_1 \land a_2) \land c_3) \land c_2) \lor c_1 \in D$
\begin{align*}
& \text{ iff }   \nnot (\nnot ((a_1 \land ( a_2 \land c_3)) \land c_2) \lor c_1 \in D & \text{by $\land$-associativity} \\
& \text{ iff }   \nnot (\nnot ((b_1 \land ( a_2 \land c_3)) \land c_2) \lor c_1 \in D & \text{by (iii)}\\
(\ldots) 
& \text{ iff }   \nnot (\nnot ((a_2 \land ( b_1 \land c_3)) \land c_2) \lor c_1 \in D \\
& \text{ iff }   \nnot (\nnot ((b_2 \land ( b_1 \land c_3)) \land c_2) \lor c_1 \in D & \text{by (iii)}\\
(\ldots) 
& \text{ iff } \nnot (\nnot ((b_1 \land b_2) \land c_3) \land c_2) \lor c_1 \in D.
\end{align*}
\end{enumerate}

 $(\lor)$. 
We have:
\begin{enumerate}[(i)]
\item $(a_1 \lor a_2) \lor c_1 \in D$
\begin{align*}
& \text{ iff } a_1 \lor (a_2 \lor c_1) \in D & \text{by $\lor$-associativity}\\
& \text{ iff } b_1 \lor (a_2 \lor c_1) \in D & \text{by (i)}\\
& \text{ iff } (b_1 \lor a_2) \lor c_1 \in D & \text{by $\lor$-associativity}\\
& \text{ iff } (a_2 \lor b_1) \lor c_1 \in D & \text{by $\lor$-commutativity}\\
& \text{ iff } a_2 \lor (b_1 \lor c_1) \in D & \text{by $\lor$-associativity}\\
& \text{ iff } b_2 \lor (b_1 \lor c_1) \in D & \text{by (i)}\\
& \text{ iff } (b_1 \lor b_2) \lor c_1 \in D & \text{by $\lor$-associativity.}
\end{align*}
\item $\quad \nnot ( (a_1 \lor a_2) \land c_2) \lor c_1 \in D$
 \begin{align*}
& \text{ iff } \nnot ( (a_1\land c_2)  \lor (a_2 \land c_2)) \lor c_1 \in D & \text{by distributivity} \\
& \text{ iff }  ( \nnot (a_1\land c_2)  \land \nnot (a_2 \land c_2)) \lor c_1 \in D & \text{by (SDM\ref{Itm:SD1})} \\
& \text{ iff }  ( \nnot (a_1\land c_2) \lor c_1)  \land (\nnot (a_2 \land c_2) \lor c_1) \in D & \text{by distributivity} \\
& \text{ iff }   \nnot (a_1\land c_2) \lor c_1,   \nnot (a_2 \land c_2) \lor c_1 \in D & \text{by }  \frac{p \land q}{p} \\
& \text{ iff }   \nnot (b_1\land c_2) \lor c_1,   \nnot (b_2 \land c_2) \lor c_1 \in D & \text{by (ii)}\\
  ( \ldots) 
& \text{ iff } \nnot ( (b_1 \lor b_2) \land c_2) \lor c_1 \in D.
\end{align*}
\item $\quad \nnot (\nnot ((a_1 \lor a_2) \land c_3) \land c_2) \lor c_1 \in D$ 
 \begin{align*}
& \text{ iff }  \nnot (\nnot ((a_1 \land c_3) \lor (a_2 \land c_3)) \land c_2) \lor c_1 \in D 
& 
\text{ by distributivity}  \\
& \text{ iff }  \nnot ( (\nnot (a_1 \land c_3) \land \nnot (a_2 \land c_3)) \land c_2) \lor c_1 \in D 
& 
\text{ by (SDM\ref{Itm:SD1}) }\\
& \text{ iff }  \nnot ( \nnot (a_1 \land c_3) \land (\nnot (a_2 \land c_3) \land c_2)) \lor c_1 \in D 
&
 \text{ by $\land$-associativity }\\
& \text{ iff }  \nnot ( \nnot (b_1 \land c_3) \land (\nnot (a_2 \land c_3) \land c_2)) \lor c_1 \in D 
& 
\text{ by (iii) }\\
& \text{ iff }  \nnot ( (\nnot (b_1 \land c_3) \land (\nnot (a_2 \land c_3)) \land c_2) \lor c_1 \in D 
& 
\text{ by $\land$-associativity }\\
& \text{ iff }  \nnot ( ((\nnot (a_2 \land c_3) \land \nnot (b_1 \land c_3)) \land c_2) \lor c_1 \in D 
& 
\text{ by $\land$-commutativity }\\
& \text{ iff }  \nnot ( ((\nnot (b_2 \land c_3) \land \nnot (b_1 \land c_3)) \land c_2) \lor c_1 \in D 
&
 \text{ by (iii) }\\
& \text{ iff }  \nnot ( (( \nnot (b_1 \land c_3) \land \nnot (b_2 \land c_3)) \land c_2) \lor c_1 \in D 
&
 \text{ by $\land$-commutativity }\\
 (\ldots) 
& \text{ iff }   \nnot (\nnot ((b_1 \lor b_2) \land c_3) \land c_2) \lor c_1 \in D. 
\end{align*}
\end{enumerate}
Hence, $\theta$ is a congruence of $\A$. Also, $\theta $ is obviously compatible with $D$.
Indeed, if $a \in D$ and $\la a,b \ra \in \theta$, then we can use 
$\frac{p}{p \lor q}$
to conclude 
$a \lor b \in D$. Then we have $b \lor b \in D$ by (i), which gives us $b \in D$ using the rule of $\lor$-idempotency.
Lastly, if $\theta' $ is a congruence of $\A$ that is compatible with $D$, then it is easy to show that $\theta' \subseteq \theta$.
Indeed, if $\la a,b \ra \in \theta'$, then  we also have, for instance, 
$
\la a \land c_2, b \land c_2 \ra, 
\la \nnot (a \land c_2), \nnot (b \land c_2) \ra, 
\la \nnot (a \land c_2) \lor c_1, \nnot (b \land c_2) \lor c_1  \ra
 \in \theta'$ and so on.
 Thus, assuming 
 $\nnot (a \land c_2) \lor c_1 \in D$, we have 
 $\nnot (b \land c_2) \lor c_1 \in D$ because $\theta'$ is compatible with $D$.
 Hence, $\la a,b \ra \in \theta$.
 Thus, $\theta$ is the largest congruence compatible with $D$, as required.
\end{proof}

The following auxiliary result is well known to hold for
semilattice-based logics (see e.g.~\cite[Thm.~2.13.iii]{AlPrRi}; for a definition of 
the classes $\Alg^*(\vdash) $ and $\Alg(\vdash) $, see
see~\cite{FJa09}).

 \begin{proposition}
\label{prop:alg}
$\Alg (\vdash^{\leq}_{\SDM}) = \SDM$. 
\end{proposition}

Table~\ref{tab:sd2} introduces the two extra rules that will permit us to axiomatize
$\vdash^{\top}_{\SDM}$. Observe that $\mathsf{r_{WP}}$ is a weaker form of 
the pseudo-complement rule $\mathsf{r_{P}}$ introduced in Example~\ref{ex:SDMPL}.
Note also that none of the rules in $\Ru_\top$ corresponds to an (in)equality:
their role is to ensure that reduced models satisfy $F=\{\top\}$, rather than to restrict the underlying class of algebras.

\begin{table}[!h] 
  \begin{tabular}{ccc}
  \\

        \AxiomC{$p \land ( \nnot (p \land q)  \lor r)$}
    \RightLabel{\, $\mathsf{r_{WP}}$ 
    }
    \UnaryInfC{$\neg q \lor r$}
        \DisplayProof &
    \AxiomC{$p \land ( \nnot (\nnot q \land r) \lor s) $}
    \RightLabel{\, $\mathsf{r_Q}$ 
    }
    \UnaryInfC{$\nnot (\nnot (p \land q) \land r) \lor s$}
    \DisplayProof \\ \\ 
    
  \end{tabular}   
\caption{The set of rules $\Ru_\top$. 
} \label{tab:sd2}
\end{table}

%
%
%
%

\begin{lemma}
\label{lem:wishfu}
Let $\la \A, F \ra$ be a reduced matrix for the strengthening of $\vdash^{\leq}_{\SDM}$ with $\Ru_\top$. 
Then $F = \{\top \}$.
\end{lemma}

\begin{proof}
By Proposition~\ref{prop:alg} (and the well-known fact that $\Alg^*(\vdash) \subseteq \Alg(\vdash)$
holds for any logic $\vdash$~\cite[Thm.~2.23]{FJa09}), we have that 
every reduced matrix for $\vdash^{\leq}_{\SDM}$ is of the form
$\la \A, F \ra$ with $\A \in \SDM$ and $F$ a lattice filter~\cite[Lemma~3.8]{Ja06a}. 

Suppose, by way of contradiction, that there is $a \in F$ such that $a \neq \top$.
Then $\la a, \top \ra \notin \Leibniz_{ \alg{A}} ( F)$.
 This means that there are $c_1, c_2, c_3 \in A$ such that at least one of the 
three items of Proposition~\ref{prop:leibn} fails. Clearly, item (i) cannot fail, because $ a, \top \in F$.
Thus, suppose item (ii)  fails. Then there are $c_1, c_2 \in A$
such that 
$\nnot ( a \land c_2) \lor c_1 \in F$ and 
$\nnot (\top \land c_2) \lor c_1 
= \nnot c_2 \lor c_1  \notin F$.
But, since $a \in F$, the latter cannot happen because of the rule $\mathsf{r_{WP}}$.
Now, assume item (iii) fails. Then there are $c_1, c_2, c_3 \in A$ such that
$ \nnot (\nnot (\top \land c_3) \land c_2) \lor c_1
= \nnot (\nnot c_3 \land c_2) \lor c_1  \in F$ but $ \nnot (\nnot (a \land c_3) \land c_2) \lor c_1 \notin F$.
But since $a \in F$, 
this  cannot happen because of rule $\mathsf{r_{Q}}$.
\end{proof}

\begin{theorem}
\label{thm:isubva}
For every $\K \subseteq \SDM$, the logic $\vdash^{\top}_{\K}$
is axiomatized, relatively to $\vdash^{\leq}_{\K}$, by $\Ru_\top$.
\end{theorem}

\begin{proof}
Soundness is clear. For completeness, assume $\Gamma \not \vdash \phi $ where $\vdash$ is 
the strengthening of $\vdash^{\leq}_{\K}$ with $\mathsf{r_{WP}}$ and $\mathsf{r_{Q}}$.
Then there is a reduced matrix model $\la \A, F \ra$ of $\vdash$
witnessing this. Moreover, $\A \in \Alg (\vdash^{\leq}_{\K}) = \VV (\K) \subseteq \SDM$
(cf.~Proposition~\ref{prop:alg}). So we can invoke Lemma~\ref{lem:wishfu} to 
obtain $F = \{ \top \}$.
Hence, $\Gamma \not \vdash^{\top}_{\K} \phi$, as required.
\end{proof}

Taking into account Theorem~\ref{thm:compl}, the preceding Theorem immediately gives us the following.

\begin{corollary}
\label{cor:allfin}
For every  $\K \subseteq \SDM$, 
if $\vdash^{\leq}_{\K}$ is finitely based, then so is
$\vdash^{\top}_{\K}$.
\end{corollary}

\section{Concluding remarks}
\label{sec:conc}

The present paper has been a contribution to improving our current  understanding of the expressivity of Hilbert calculi. 
As observed earlier, 
Gentzen calculi allow one to impose directly the meta-properties  needed to ensure that the inter-derivability relation is a congruence
of the formula algebra.
By contrast, we have shown that  under certain conditions
this is beyond what Hilbert calculi can capture finitely. 
Our main
results
 are displayed in Table~\ref{tab:sum} below. 

\begin{table}[!h]

\begin{center}
\begin{tabular}{ c | c | c | c }
Conditions on $\Eq$  & $\vdash^\leq_{\VV_\Eq}$ & $\vdash^\top_{\VV_\Eq}$ & Examples
\smallskip\\ 
  \hline
$\neg$-balanced, $\VV_\Eq\subseteq \DN$ 
and 
$\neg^k p \,\not\vdash^{\leq}_{\VV_\Eq} \neg^k q$
 & ${\bf N}$ & ? & $\DN$, $\OC$\\  
 finite and $\VV_\Eq\subseteq \SDM$ & ${\bf Y}$ & ${\bf Y}$ & $\SDM$, $\PDL$   \\
  finite and $\VV_\Eq\subseteq \OC^m_n$ & ${\bf Y} $& ? &     $\OC^m_n$, $\DM$\\
\end{tabular}
\end{center}

 \caption{Finite axiomatizability results.} \label{tab:sum}
 \end{table}

On the front of positive results, we have identified certain subvarieties of $\DN$
for which Hilbert calculi are indeed able to reflect finitely the effect of 
 imposing extra equations on the algebras.
The well-known result that finitely-generated varieties of lattices
are finitely based~\cite[Cor.~V.4.18]{BuSa00} implies that our methods
may be successfully applied to every finite-valued order-preserving logic
that extends $\vdash^{\leq}_{\SDM}$. We believe it would be interesting to 
take a closer look at the conditions that characterize this divide.

Yet another  approach to the axiomatization of logics, which is intermediate 
between Hilbert and Gentzen, is provided by multiple-conclusion calculi.
These are an extension of traditional (single-conclusion) Hilbert calculi 
where rules may have non-singleton sets of conclusions (which are read disjunctively).
With
multiple-conclusion calculi one gains a considerably greater expressive power without 
expanding the signature with metalinguistic symbols as happens with Gentzen systems.
For instance, it is known that  every finite-valued 
logic is finitely axiomatizable by multiple-conclusion calculi, and 
 desirable proof-theoretical properties (e.g.~analiticity, effective proof search) are more easily established for the latter
than for the their single-conclusion counterparts (see e.g.~\cite{ShSm78,CaMa19,CaMaXX}).
We speculate whether the logics we have shown to be non-finitely based (by means of single-conclusion Hilbert calculi)
might be axiomatizable by means of a finite multiple-conclusion calculus (as happens, for instance,
with the logic defined by Wro\'nski's three-element matrix: see \cite{Wr79,CaMaXX}).

%
A related question is whether logics of distributive lattices with negation that are not given by any
finite set of finite matrices
may admit some finite non-deterministic partial matrix semantics (see~\cite{Avron05,Baaz13,charfinval,soco,CaMaAXS}). 

A last research direction worth mentioning is the study of logics defined from classes  of distributive lattices
with negation through different choices of the designated elements. 
As we have seen earlier, one such choice yields $\top$-assertional logics associated to subvarieties of
$\DN$. 
In this respect, we speculate whether the finite axiomatizability result
obtained in Section~\ref{sec:top} for
$\vdash^{\top}_{\SDM}$ might be extended to
other logics (e.g.~$\vdash^{\top}_{\DN}, \vdash^{\top}_{\OC}, \vdash^{\top}_{\OC^m_n}$).

\end{document}